\documentclass[12pt]{amsart}

\usepackage{amsmath}
\usepackage{amsfonts}
\usepackage{amssymb}
\usepackage[all]{xy}           
\usepackage{xcolor}
\usepackage{bbding}
\usepackage{txfonts}
\usepackage{amscd}

\usepackage[shortlabels]{enumitem}
\usepackage{ifpdf}
\ifpdf
  \usepackage[colorlinks,final,
  hyperindex]{hyperref}
\else
  \usepackage[colorlinks,final,
  hyperindex]{hyperref}
\fi
\usepackage{tikz}
\usepackage[active]{srcltx}

\makeatletter

\newtheorem{thm}{Theorem}[section]
\newtheorem{lem}[thm]{Lemma}

\newtheorem{pro}[thm]{Proposition}
\newtheorem{ex}[thm]{Example}
\newtheorem{rmk}[thm]{Remark}
\newtheorem{defi}[thm]{Definition}

\numberwithin{equation}{section}

\newcommand{\g}{\mathfrak {g}}
\newcommand{\gl}{\mathrm{End}}
\newcommand{\kl}{\mathfrak l}
\newcommand{\kr}{\mathfrak r}
\newcommand{\kkl}{\tilde{\mathfrak{l}}}
\newcommand{\kkr}{\tilde{\mathfrak{r}}}
\newcommand{\kll}{\bar{\mathfrak{l}}}
\newcommand{\krr}{\bar{\mathfrak{r}}}

\newcommand{\fl}{\mathbf{l}}
\newcommand{\fr}{\mathbf{r}}

\newcommand{\ffl}{\tilde{\mathbf{l}}}
\newcommand{\ffr}{\tilde{\mathbf{r}}}
\newcommand{\fll}{\bar{\mathbf{l}}}
\newcommand{\frr}{\bar{\mathbf{r}}}

\newcommand{\id}{\mathrm{id}}

\newcommand{\ad}{\mathrm{ad}}
\newcommand{\Img}{\mathrm{Im}}

\newcommand{\DAYBE}{\mathrm{DAYBE}}
\newcommand{\CYBE}{\mathrm{CYBE}}
\newcommand{\AYBE}{\mathrm{AYBE}}
\newcommand{\LYBE}{\mathrm{LYBE}}

\begin{document}

\title[Construction of diassociative bialgebras from ASI bialgebras]
{Construction of diassociative bialgebras from antisymmetric infinitesimal bialgebras
and related algebra structures}

\author{Bo Hou}
\address{School of Mathematics and Statistics, Henan University, Kaifeng 475004,
China}
\email{houbo@henu.edu.cn, bohou1981@163.com}

\author{Ru Li}
\address{School of Mathematics and Statistics, Henan University, Kaifeng 475004,
China}
\email{13037698973@163.com}

\vspace{-5mm}


\begin{abstract}
There is a diassociative algebra structure on the tensor product of an associative algebra
and a perm algebra. In this paper, we elevate this conclusion to the level of bialgebra.
We prove that the tensor product of an antisymmetric infinitesimal bialgebra
and a quadratic perm algebra has a diassociative bialgebra structure, and this diassociative
bialgebra structure is coboundary (resp. quasi-triangular, triangular, factorizable) if
the original antisymmetric infinitesimal bialgebra is coboundary (resp. quasi-triangular,
triangular, factorizable). As an application, we provide the close relationship between
Lie bialgebras, Leibniz bialgebras, diassociative bialgebras and antisymmetric
infinitesimal bialgebras, and the close relationship between symplectic Lie algebras,
symplectic Leibniz algebras, symplectic associative algebras and symplectic
diassociative algebras.
\end{abstract}

\makeatletter
\@namedef{subjclassname@2020}{\textup{2020} Mathematics Subject Classification}
\makeatother

\maketitle



\vspace{-4mm}

\section{Introduction}\label{sec:intr}

The aim of this paper is to construct diassociative bialgebras by the antisymmetric
infinitesimal bialgebras. As a direct application of this construction, we present the
close connection between Lie bialgebras, Leibniz bialgebras, diassociative bialgebras
and antisymmetric infinitesimal bialgebras.

A bialgebra structure on a given algebra structure is obtained as a coalgebra
structure together which gives the same algebra structure on the dual space with
a set of compatibility conditions. One of the most famous example of bialgebra is
the Lie bialgebra \cite{Dri}. Lie bialgebras have important applications in various
areas, including the theory of integrable systems and the study of Poisson Lie groups
\cite{CP}. In the context of Lie bialgebras, quasi-triangular Lie bialgebras play important
roles in mathematical physics \cite{Sem}. In the past decade or so, there have been a lot
of bialgebra theories for other algebra structures that essentially follow the approach
of Lie bialgebras such as antisymmetric infinitesimal bialgebras \cite{Agu,Bai},
left-symmetric bialgebras \cite{Bai1}, Jordan bialgebras \cite{Zhe}, Novikov bialgebras
\cite{HBG}, perm bialgebras \cite{Hou,LZB}, Jacobi-Jordan bialgebras \cite{BCHM},
anti-pre-Lie Poisson bialgebras \cite{LB}, and so on.
Recently, \cite{Lu} and \cite{HLLZ} independently presented a bialgebra theory of
diassociative algebras. In this paper, we will present a method for constructing
diassociative bialgebras by using the tensor product of antisymmetric infinitesimal bialgebras
(or ASI bialgebras) and quadratic perm algebras. And on the basis of this method,
the close connection between diassociative bialgebras and Lie bialgebras, Leibniz
bialgebras are discussed in detail.

As an important class of quasi-triangular Lie bialgebras, factorizable Lie bialgebras
are used to connect classical $r$-matrices with certain factorization problems, and have
various applications in integrable systems \cite{BGN,RS}. Recently, the factorizable Lie
bialgebras have received further research \cite{LS}, and the factorizable antisymmetric
infinitesimal bialgebras \cite{SW}, factorizable Leibniz bialgebras \cite{BLST},
factorizable diassociative bialgebras \cite{Lu} have been studied. In this paper,
we consider the factorizability of the diassociative bialgebra induced by an
antisymmetric infinitesimal bialgebra.

\smallskip\noindent
{\bf Theorem } (Theorems \ref{thm:permbia-dia} and \ref{thm:indu-sdibia})
{\it Let $(A, \cdot, \Delta)$ be an ASI bialgebra and $(P, \diamond, \omega)$ be a
quadratic perm algebra and $(A\otimes P, \dashv, \vdash)$ be the induced diassociative
algebra, i.e., $(a_{1}, p_{1})\dashv(a_{2}, p_{2})=(a_{1}\cdot a_{2})\otimes(p_{1}
\diamond p_{2})$ and $(a_{1}, p_{1})\vdash(a_{2}, p_{2})=(a_{1}\cdot a_{2})\otimes
(p_{2}\diamond p_{1})$ for any $a_{1}, a_{2}\in A$ and $p_{1}, p_{2}\in P$. Define two
linear maps $\Delta_{\dashv}, \Delta_{\vdash}: A\otimes P\rightarrow
(A\otimes P)\otimes(A\otimes P)$ by
\begin{align*}
\Delta_{\dashv}(a\otimes p)&=\Delta(a)\bullet\tau(\nu_{\varpi}(p))=\sum_{(a)}\sum_{(p)}
(a_{(1)}\otimes p_{(2)})\otimes(a_{(2)}\otimes p_{(1)}),\\[-2mm]
\Delta_{\vdash}(a\otimes p)&=\Delta(a)\bullet\nu_{\varpi}(p)=\sum_{(a)}\sum_{(p)}
(a_{(1)}\otimes p_{(1)})\otimes(a_{(2)}\otimes p_{(2)}),
\end{align*}
for any $a\in A$ and $p\in P$, where $\nu_{\varpi}(p)=\sum_{(p)}p_{(1)}\otimes p_{(2)}$
and $\Delta(a)=\sum_{(a)}a_{(1)}\otimes a_{(2)}$ in the Sweedler notation.
Then $(A\otimes P, \dashv, \vdash, \Delta_{\dashv}, \Delta_{\vdash})$ is a diassociative
bialgebra. Moreover, we get that $(A\otimes P, \dashv, \vdash$, $\Delta_{\dashv},
\Delta_{\vdash})$ is a coboundary (resp. quasi-triangular, triangular, factorizable)
diassociative bialgebra if $(A, \cdot, \Delta)$ is coboundary (resp. quasi-triangular,
triangular, factorizable).}

\smallskip
In \cite{Kup}, Kupershmidt found that the classical Yang-Baxter equation (or $\CYBE$)
in tensor form on Lie algebras can be converted into an $\mathcal{O}$-operator associated
to the coregular representation. This conclusion has been proven to be valid for various
algebra structures. Nowadays, the $\mathcal{O}$-operator is regarded as an operator form
of a skew-symmetric solution of the $\CYBE$. On the other hand, a skew-symmetric solution of
the $\CYBE$ in a Lie algebra naturally induces a triangular Lie bialgebra. In this paper,
by discussing the connection between the solution of the associative Yang-Baxter equation
in an associative algebra $(A, \cdot)$ and the solution of the $\CYBE$ in the induced
Lie algebra $(A, [-,-])$, the solution of the diassociative Yang-Baxter equation in the
related diassociative $(A\otimes P, \dashv, \vdash)$, as well as the solution of the
Leibniz Yang-Baxter equation in the related Leibniz algebra $(A\otimes P, \ast)$,
we present the close connections among triangular antisymmetric infinitesimal bialgebras
triangular Lie bialgebras, triangular diassociative bialgebras and triangular Leibniz
bialgebras, as well as the connections among $\mathcal{O}$-operators on associative
algebras, Lie algebras, diassociative algebras and Leibniz algebras.
These results, when put together, yield the following commutative diagram:
{\small
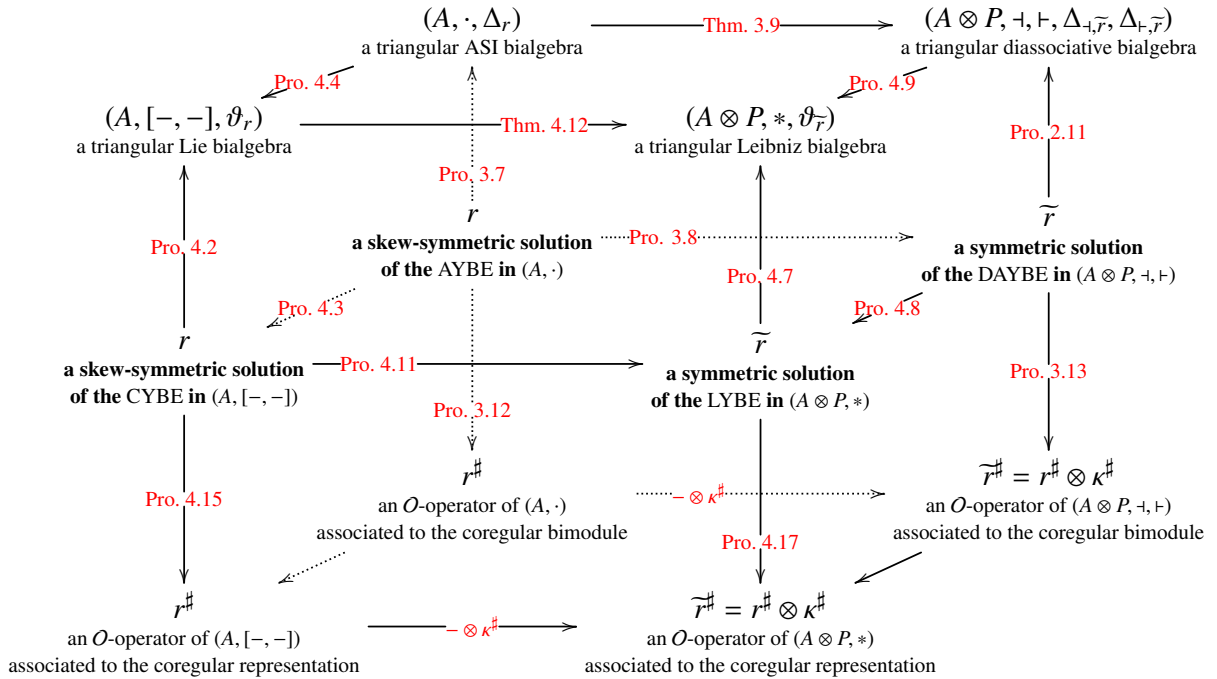
\begin{displaymath}
\xymatrix@R=0.4cm@C=-0.8cm{
&\txt{{\small $(A, \cdot, \Delta_{r})$ }\\
{\tiny a triangular ASI bialgebra}}
\ar@{->}[rr]|-{\txt{\tiny\textcolor{red}{Thm.}~\ref{thm:indu-sdibia}}}
\ar@{->}[ld]|-{\txt{\tiny\textcolor{red}{Pro.}~\ref{pro:qtAss-qtLie}}}
\ar@{<.}[dd]|-(0.8){\txt{\tiny\textcolor{red}{Pro.}~\ref{pro:quasass-bia}}}&
&\txt{{\small $(A\otimes P, \dashv, \vdash, \Delta_{\dashv,\widetilde{r}},
\Delta_{\vdash,\widetilde{r}})$} \\
{\tiny a triangular diassociative bialgebra}}
\ar@{->}[ld]|-{\txt{\tiny\textcolor{red}{Pro.}~\ref{pro:qtdiAss-qtLeib}}}
\ar@{<-}[dd]|-{\txt{\tiny\textcolor{red}{Pro.}~\ref{pro:quasi-di}}}&\\
\txt{{\small $(A, [-,-], \vartheta_{r})$} \\ {\tiny a triangular Lie bialgebra}}
\ar@{<-}[dd]|-{\txt{\tiny\textcolor{red}{Pro.}~\ref{pro:lie-bia}}}
\ar@{->}[rr]|-(0.75){\txt{\tiny\textcolor{red}{Thm.}~\ref{thm:indu-triLib}}}&
&\txt{{\small $(A\otimes P, \ast, \vartheta_{\widetilde{r}})$}\\
{\tiny a triangular Leibniz bialgebra}}
\ar@{<-}[dd]|-(0.7){\txt{\tiny\textcolor{red}{Pro.}~\ref{pro:sLib-bia}}} \\
&\txt{{\small\bf $r$}\\ {\tiny\bf a skew-symmetric solution} \\
{\tiny\bf of the $\AYBE$ in $(A, \cdot)$}}
\ar@{.>}[dd]|-(0.75){\txt{\tiny\textcolor{red}{Pro.}~\ref{pro:o-ass}}}
\ar@{.>}[rr]|-(0.2){\rm\textcolor{red}{Pro.}~\ref{pro:AYBE-DYBE}}&
&\txt{{\small\bf $\widetilde{r}$} \\  {\tiny\bf a symmetric solution} \\
{\tiny\bf of the $\DAYBE$ in $(A\otimes P, \dashv, \vdash)$}}
\ar@{->}[dd]|-{\txt{\tiny\textcolor{red}{Pro.}~\ref{pro:o-dia}}}\\
\txt{{\small\bf $r$}\\ {\tiny\bf a skew-symmetric solution} \\
{\tiny\bf of the $\CYBE$ in $(A, [-,-])$}}
\ar@{->}[dd]|-{\txt{\tiny\textcolor{red}{Pro.}~\ref{pro:o-lie}}}
\ar@{<.}[ru]|-{\txt{\tiny\textcolor{red}{Pro.}~\ref{pro:ass-Lie-YBE}}}
\ar@{->}[rr]|-(0.2){\txt{\tiny\textcolor{red}{Pro.}~\ref{pro:CYBE-LYBE}}}&
&\txt{{\small\bf $\widetilde{r}$}\\ {\tiny\bf a symmetric solution} \\
{\tiny\bf of the $\LYBE$ in $(A\otimes P, \ast)$}}
\ar@{<-}[ru]|-{\txt{\tiny\textcolor{red}{Pro.}~\ref{pro:diass-Leib-YBE}}}
\ar@{->}[dd]|-(0.75){\txt{\tiny\textcolor{red}{Pro.}~\ref{pro:o-leib}}}\\
&\txt{{\small $r^{\sharp}$}\\{\tiny an $\mathcal{O}$-operator of $(A, \cdot)$} \\
{\tiny associated to the coregular bimodule}} \ar@{.>}[ld] &
&\txt{{\small $\widetilde{r}^{\sharp}=r^{\sharp}\otimes\kappa^{\sharp}$}\\
{\tiny an $\mathcal{O}$-operator of $(A\otimes P, \dashv, \vdash)$} \\
{\tiny associated to the coregular bimodule}}
\ar@{<.}[ll]|-(0.75){\txt{\tiny\textcolor{red}{$-\otimes\kappa^{\sharp}$}}}\\
\txt{{\small $r^{\sharp}$}\\{\tiny an $\mathcal{O}$-operator of $(A, [-,-])$} \\
{\tiny associated to the coregular representation}}&
&\txt{{\small $\widetilde{r}^{\sharp}=r^{\sharp}\otimes\kappa^{\sharp}$}\\
{\tiny an $\mathcal{O}$-operator of $(A\otimes P, \ast)$} \\
{\tiny associated to the coregular representation}}
\ar@{<-}[ru]\ar@{<-}[ll]|-{\txt{\tiny\textcolor{red}{$-\otimes\kappa^{\sharp}$}}}&}
\end{displaymath}
}

A symplectic structure on a Lie algebra $(\g, [-,-])$ is a nondegenerate 2-cocycle
\cite{AS,GMAB}. The underlying structure of a symplectic Lie algebra is a quadratic
pre-Lie algebra. Moreover, the symplectic structure on a Lie algebra is closely related
to the nondegenerate skew-symmetric solution of the $\CYBE$ in this Lie algebra.
A symplectic associative algebra is an associative algebra with a nondegenerate Connes
cocycle. Recently, Tang, Xu and Sheng provided a definition of symplectic Leibniz algebras
and studied the symplectic structures and the complex structures on a Leibniz algebras
\cite{TXS}. Here we introduce the definition of symplectic diassociative algebras
and obtain a commutative diagram about symplectic associative algebra, symplectic
diassociative algebra, symplectic Lie algebra and symplectic Leibniz algebra
(see Theorem \ref{thm:sym-ass-leib}):
$$
\xymatrix@C=2cm@R=0.7cm{
\txt{$(A, \cdot, \omega)$ \\
{\tiny a symplectic associative algebra}}
\ar[d]_{{\rm Pro.}~\ref{pro:sym-ass-lie}} \ar[r]^{{\rm Pro.}~\ref{pro:sym-ass-diass}}
&\txt{$(A\otimes P, \dashv, \vdash, \widetilde{\omega})$\\
{\tiny a symplectic diassociative algebra}}\ar[d]^{{\rm Pro.}~\ref{pro:sym-diass-leib}} \\
\txt{$(A, [-,-], \omega)$ \\ {\tiny a symplectic Lie algebra}}
\ar[r]^{{\rm Pro.}~\ref{pro:sym-lie-leib}}
& \txt{$(A\otimes P, \ast, \widetilde{\omega})$ \\ {\tiny a symplectic Leibniz algebra}}}
$$
Moreover, we also use the close relationship between symplectic structures and the
nondegenerate solutions of the Yang-Baxter equations to provide the construction
relationship between symplectic associative algebras, symplectic Lie algebras,
symplectic diassociative algebras and symplectic Leibniz algebras.

This paper is organized as follows. In Section \ref{sec:Pre}, we recall the notions of
associative algebras, diassociative algebras and diassociative bialgebras. We give a
construction of Leibniz algebras from associative algebras.
In Section \ref{sec:qtbia-ASI}, we show that there is a diassociative bialgebra on the
tensor product of an antisymmetric infinitesimal bialgebra and a quadratic perm algebra
in Theorem \ref{thm:permbia-dia}, and prove that the induced diassociative bialgebra
is coboundary (resp. quasi-triangular, triangular, factorizable) if the antisymmetric
infinitesimal bialgebra is coboundary (resp. quasi-triangular, triangular, factorizable)
in Theorem \ref{thm:indu-sdibia}. Moreover, we also give a construction of
$\mathcal{O}$-operator of diassociative bialgebras from $\mathcal{O}$-operator of
associative algebras. In Section \ref{sec:LeibBi}, We use a series of propositions to
provide the connection between the solution of the associative Yang-Baxter equation
in an associative algebra $(A, \cdot)$ and the solution of the $\CYBE$ in the induced
Lie algebra $(A, [-,-])$, the solution of the diassociative Yang-Baxter equation in the
related diassociative $(A\otimes P, \dashv, \vdash)$, as well as the solution of the
Leibniz Yang-Baxter equation in the related Leibniz algebra $(A\otimes P, \ast)$.
Furthermore, the commutative diagram mentioned above about solutions of the
Yang-Baxter equations, the triangular bialgebra structures and the $\mathcal{O}$-operators
is provided. Moreover, by using the connection between the nondegenerate solution of the
Yang-Baxter equation and symplectic structure, the construction relationship between
symplectic associative algebras, symplectic Lie algebras, symplectic diassociative
algebras and symplectic Leibniz algebras is presented.

Throughout this paper, we fix $\Bbbk$ as a field of characteristic zero.
All the vector spaces, algebras are over $\Bbbk$ and are finite-dimensional
unless otherwise specified, and all tensor products are also over $\Bbbk$.
We denote the identity map by $\id$. For any finite-dimensional $\Bbbk$-vector
space $V$, we denote $V^{\ast}$ the dual space.

\section{Preliminaries}\label{sec:Pre}
In this section, we recall the notions of associative algebras, diassociative algebras
and diassociative bialgebras.

\subsection{Diassociative algebras and bimodules over a diassociative algebra}
\label{subsec:diass}
Recall that an {\bf associative algebra} $(A, \cdot)$ is a vector space $A$ equipped with
a binary operation $\cdot: A\otimes A\rightarrow A$ such that for any $a_{1}, a_{2}, a_{3}
\in A$,
$$
a_{1}\cdot(a_{2}\cdot a_{3})=(a_{1}\cdot a_{2})\cdot a_{3}.
$$
Diassociative algebra is a generalization of associative algebra, which was first
introduced by Loday in the early 1990s in connection with  periodicity phenomena in
algebraic $K$-theory.

\begin{defi}\label{de:dialg}
A {\bf diassociative algebra} $(D, \dashv, \vdash)$ is a vector space $D$ equipped with
two bilinear maps called respectively left product and right product: $\dashv, \vdash:
D\otimes D\rightarrow D$, such that $(D, \dashv)$ and $(D, \vdash)$
are associative algebras and satisfying the following axioms:
\begin{align*}
& (d_{1}\dashv d_{2})\dashv d_{3}=d_{1}\dashv(d_{2}\vdash d_{3}),\\
& (d_{1}\vdash d_{2})\dashv d_{3}=d_{1}\vdash(d_{2}\dashv d_{3}),\\
& (d_{1}\dashv d_{2})\vdash d_{3}=(d_{1}\vdash d_{2})\vdash d_{3},
\end{align*}
for any $d_{1}, d_{2}, d_{3}\in D$
\end{defi}

Recall that a (left) {\bf perm algebra} is a pair $(P, \diamond)$, where $P$ is a vector
space and $\diamond: P\otimes P\rightarrow P$ is a binary operator such that for any
$p_{1}, p_{2}, p_{3}\in P$,
$$
p_{1}\diamond(p_{2}\diamond p_{3})=(p_{1}\diamond p_{2})\diamond p_{3}
=(p_{2}\diamond p_{1})\diamond p_{3}.
$$
Clearly, in a diassociative algebra $(D, \dashv, \vdash)$, if $d_{1}\dashv d_{2}=
d_{2}\vdash d_{1}$ for any $d_{1}, d_{2}\in D$, we get that $(D, \dashv)$ is a perm algebra.
Let $(A, \cdot)$ be an associative algebra. If we define bilinear maps
$\dashv, \vdash: A\otimes A\rightarrow A$ by $a_{1}\dashv a_{2}=a_{1}\cdot a_{2}
=a_{1}\vdash a_{2}$ for any $a_{1}, a_{2}\in A$. Then $(A, \dashv, \vdash)$ is a
diassociative algebra. That is to say, we can view each associative algebra as a
diassociative algebra in this way. Moreover, for any associative algebra, we
can also use its tensor product with the perm algebra to obtain a diassociative
algebra.

\begin{pro}[\cite{Lod}]\label{pro:ass-diass}
Let $(A, \cdot)$ be an associative algebra and $(P, \diamond)$ be a perm algebra.
If we define two binary operations $\vdash, \dashv$ on $A\otimes P$ by
\begin{align}
&(a_{1}\otimes p_{1})\dashv(a_{2}\otimes p_{2})
:=(a_{1}\cdot a_{2})\otimes(p_{2}\diamond p_{1}),  \label{prod1}\\
&(a_{1}\otimes p_{1})\vdash(a_{2}\otimes p_{2})
:=(a_{1}\cdot a_{2})\otimes(p_{1}\diamond p_{2}),  \label{prod2}
\end{align}
for any $a_{1}, a_{2}\in A$ and $p_{1}, p_{2}\in P$, then $(A\otimes P, \vdash, \dashv)$ is
a diassociative algebra.
\end{pro}

Recall that a (left) {\bf Leibniz algebra} $(L, \ast)$ is a vector space $L$ together
with a bilinear map $\ast: L\times L\rightarrow L$ satisfying the following Leibniz identity:
$$
x_{1}\ast(x_{2}\ast x_{3})=(x_{1}\ast x_{2})\ast x_{3}+x_{2}\ast(x_{1}\ast x_{3}),
$$
for any $x_{1}, x_{2}, x_{3}\in L$. In a Leibniz algebra $(L, \ast)$, if $\ast$ is
anticommutative, i.e., $x_{1}\ast x_{2}=-x_{2}\ast x_{1}$, then $(L, \ast)$ is a
{\bf Lie algebra}.
In a Lie algebra, we usually denote the binary operation $\ast$ by bracket $[-,-]$.
It is well-known that an associative algebra forms a Lie algebra under the commutator.
Similar results also exist for diassociative algebras and Leibniz algebras.

\begin{pro}\label{pro:commtor}
Let $(D, \dashv, \vdash)$ be a diassociative algebra. Define a new binary operation $\ast$
on $D$ by
$$
d_{1}\ast d_{2}=d_{1}\vdash d_{2}-d_{2}\dashv d_{1},
$$
for any $d_{1}, d_{2}\in D$, then $(D, \ast)$ is a Leibniz algebra. In particular,
if $(D, \cdot)$ is an associative algebra, then $(D, [-,-])$ is a Lie algebra,
where $[d_{1}, d_{2}]=d_{1}\cdot d_{2}-d_{2}\cdot d_{1}$.
\end{pro}

Direct verification shows that there is a Leibniz algebra structure on the tensor product
of a Lie algebra and a perm algebra:

\begin{pro}\label{pro:li-lieb}
Let $(\g, [-,-])$ be a Lie algebra and $(P, \diamond)$ be a perm algebra.
If we define a binary operation $\ast$ on $\g\otimes P$ by
\begin{align}
&(g_{1}\otimes p_{1})\ast(g_{2}\otimes p_{2})
:=[g_{1}, g_{2}]\otimes(p_{1}\diamond p_{2}),  \label{prod3}
\end{align}
for any $g_{1}, g_{2}\in\g$ and $p_{1}, p_{2}\in P$.
Then $(\g\otimes P, \ast)$ is a Leibniz algebra.
\end{pro}

For a given associative algebra $(A, \cdot)$ and a perm algebra $(P, \diamond)$,
one can check that the Leibniz algebra $(\g\otimes P, \ast)$ constructed in Proposition
\ref{pro:li-lieb} is precisely the the induced Leibniz algebra by the diassociative algebra
$(A\otimes P, \vdash, \dashv)$ given in Proposition \ref{pro:ass-diass}.
Thus, we obtain the following commutative diagram:
$$
\xymatrix@C=8cm@R=0.5cm{
\txt{$(A, \cdot)$ \\ {\tiny an associative algebra}}
\ar[d]_-{{\rm Pro.}~\ref{pro:ass-diass}}
\ar[r]^-{\mbox{\tiny $[a_{1}, a_{2}]=a_{1}\cdot a_{2}-a_{2}\cdot a_{1}$}}
&\txt{$(A, [-,-])$\\ {\tiny a Lie algebra}}
\ar[d]^-{{\rm Pro.}~\ref{pro:li-lieb}} \\
\txt{$(A\otimes P, \dashv, \vdash)$ \\ {\tiny a diassociative algebra}}
\ar[r]^-{\mbox{\tiny $(a_{1}\otimes p_{1})\ast(a_{2}\otimes p_{2})=(a_{1}\otimes p_{1})
\vdash(a_{2}\otimes p_{2})-(a_{2}\otimes p_{2})\dashv(a_{1}\otimes p_{1})$}}
& \txt{$(A\otimes P, \ast)$ \\ {\tiny a Leibniz algebra}}}
$$
One of the main conclusions of this paper is that this commutative diagram also holds
at the level of bialgebra structure and symplectic structure.

\begin{rmk}\label{rmk:comm}
Since each associative algebra is a diassociative algebra and each Lie algebra is a
Leibniz algebra, we naturally have the following commutative diagram:
$$
\xymatrix@C=2cm@R=0.5cm{
\txt{$(A, \cdot)$ \\ {\tiny an associative algebra}}
\ar[d]\ar[r]^-{\mbox{\tiny\rm commutator}}
&\txt{$(A, [-,-])$\\ {\tiny a Lie algebra}}\ar[d] \\
\txt{$(A, \cdot, \cdot)$ \\ {\tiny a diassociative algebra}}
\ar[r]^-{\mbox{\tiny\rm commutator}}
& \txt{$(A, \ast=[-,-])$ \\ {\tiny a Leibniz algebra}}}
$$
But an antisymmetric infinitesimal bialgebra is not a diassociative bialgebra in general
 and a Lie bialgebra is not a Leibniz bialgebra in general.
Thus, this commutative diagram not holds at the level of bialgebra.
\end{rmk}

Next, we consider bimodules over a diassociative algebra. Let $(A, \cdot)$ be an associative
algebra $V$ be a vector space and $\kl, \kr: A\rightarrow\gl(V)$ be two linear maps.
If for any $a_{1}, a_{2}\in A$, $\kl(a_{1}\cdot a_{2})=\kl(a_{1})\circ\kl(a_{2})$,
$\kl(a_{1})\circ\kr(a_{2})=\kl(a_{2})\circ\kl(a_{1})$ and $\kr(a_{1}\cdot a_{2})
=\kr(a_{2})\circ\kr(a_{1})$, then $(V, \kl, \kr)$ is called a {\bf bimodule} over
$(A, \cdot)$. In particular, $(A, \fl_{A}, \fr_{A})$ is a bimodule over $(A, \cdot)$,
which is celled the {\bf regular bimodule} over $(A, \cdot)$, where $\fl_{A}(a_{1})(a_{2})
=a_{1}\cdot a_{2}=\fr_{A}(a_{2})(a_{1})$ for any $a_{1}, a_{2}\in A$.

\begin{defi}\label{def:SL-mod}
Let $(D, \dashv, \vdash)$ be a diassociative algebra, $V$ be a $\Bbbk$-vector space.
If there exist four bilinear maps $\kl_{\dashv}, \kr_{\dashv}, \kl_{\vdash}, \kr_{\vdash}:
D\rightarrow\gl(V)$, such that $(V, \kl_{\dashv}, \kr_{\dashv})$ is a bimodule
over $(D, \dashv)$, $(V, \kl_{\vdash}, \kr_{\vdash})$ is a bimodule over $(D, \vdash)$
and for any $d_{1}, d_{2}\in D$,
\begin{align*}
\kr_{\vdash}(d_{1}\vdash d_{2})&=\kr_{\vdash}(d_{1}\dashv d_{2}),
&& \kl_{\vdash}(d_{1})\circ\kr_{\vdash}(d_{2})=\kl_{\vdash}(d_{1})\circ\kr_{\dashv}(d_{2}),
&& \kl_{\vdash}(d_{1})\circ\kl_{\vdash}(d_{2})=\kl_{\vdash}(d_{1})\circ\kl_{\dashv}(d_{2}),  \\
\kr_{\vdash}(d_{1})\circ\kr_{\dashv}(d_{2})&=\kr_{\dashv}(d_{2}\vdash d_{1}),
&& \kr_{\vdash}(d_{1})\circ\kl_{\dashv}(d_{2})=\kl_{\dashv}(d_{2})\circ\kr_{\vdash}(d_{1}),
&&\; \kl_{\vdash}(d_{1}\dashv d_{2})=\kl_{\dashv}(d_{1})\circ\kl_{\vdash}(d_{2}),  \\
\kr_{\dashv}(d_{1})\circ\kr_{\vdash}(d_{2})&=\kr_{\dashv}(d_{1})\circ\kr_{\dashv}(d_{2}),
&& \kr_{\dashv}(d_{1})\circ\kl_{\vdash}(d_{2})=\kr_{\dashv}(d_{1})\circ\kl_{\dashv}(d_{2}),
&&\; \kl_{\dashv}(d_{1}\vdash d_{2})=\kl_{\dashv}(d_{1}\dashv d_{2}),
\end{align*}
we call $(V, \kl_{\dashv}, \kr_{\dashv}, \kl_{\vdash}, \kr_{\vdash})$ is a
{\bf bimodule} over $(D, \dashv, \vdash)$.
\end{defi}

Let $(V, \kl_{\dashv}, \kr_{\dashv}, \kl_{\vdash}, \kr_{\vdash})$ and $(V', \kl'_{\dashv},
\kr'_{\dashv}, \kl'_{\vdash}, \kr'_{\vdash})$ be two bimodules over a diassociative algebra
$(D, \dashv, \vdash)$ and $f: V\rightarrow V'$ be a linear map. Then $f$ is called a
{\bf morphism of bimodule} if $f(\beta(d)(v))=\beta'(d)(f(v))$ for any
$\beta\in\{\kl_{\dashv}, \kr_{\dashv}, \kl_{\vdash}, \kr_{\vdash}\}$, $d\in D$ and $v\in V$.
This two bimodules $(V, \kl_{\dashv}, \kr_{\dashv}, \kl_{\vdash}, \kr_{\vdash})$ and
$(V', \kl'_{\dashv}, \kr'_{\dashv}, \kl'_{\vdash}, \kr'_{\vdash})$ are called {\bf
isomorphic as bimodules} if the morphism $f$ is a bijection. It is easy to see that
a diassociative algebra $(D, \dashv, \vdash)$ is a bimodule over itself under the actions
$\fl_{\dashv}, \fr_{\dashv}, \fl_{\vdash}, \fr_{\vdash}: D\rightarrow\gl(D)$,
$\fl_{\dashv}(d_{1})(d_{2})=d_{1}\dashv d_{2}$,
$\fr_{\dashv}(d_{1})(d_{2})=d_{2}\dashv d_{1}$,
$\fl_{\vdash}(d_{1})(d_{2})=d_{1}\vdash d_{2}$ and
$\fr_{\vdash}(d_{1})(d_{2})=d_{2}\vdash d_{1}$ for $d_{1}, d_{2}\in D$.
This bimodule is called the {\bf regular bimodule} over $(D, \dashv, \vdash)$.
Moreover, by direct calculations, we can give an equivalent condition as follows.

\begin{pro}\label{pro:bimodule}
Let $(D, \dashv, \vdash)$ be a diassociative algebra, $V$ be a $\Bbbk$-vector space,
$\kl_{\dashv}, \kr_{\dashv}, \kl_{\vdash}, \kr_{\vdash}: D\rightarrow\gl(V)$ be four
linear maps. Then $(V, \kl_{\dashv}, \kr_{\dashv}, \kl_{\vdash}, \kr_{\vdash})$ is a
bimodule over $(D, \dashv, \vdash)$ if and only if $D\oplus V$ is a diassociative algebra under
the following operations:
\begin{align*}
(d_{1}, v_{1})\dashv(d_{2}, v_{2})&:=\big(d_{1}\dashv d_{2},\ \
\kl_{\dashv}(d_{1})(v_{2})+\kr_{\dashv}(d_{2})(v_{1})\big),\\
(d_{1}, v_{1})\vdash(d_{2}, v_{2})&:=\big(d_{1}\vdash d_{2},\ \
\kl_{\vdash}(d_{1})(v_{2})+\kr_{\vdash}(d_{2})(v_{1})\big),
\end{align*}
for all $d_{1}, d_{2}\in D$ and $v_{1}, v_{2}\in V$. This diassociative algebra is called
a {\bf semidirect product} of $(D, \dashv, \vdash)$ by bimodule $(V, \kl_{\dashv},
\kr_{\dashv}, \kl_{\vdash}, \kr_{\vdash})$, denoted by $D\ltimes V$.
\end{pro}

\begin{proof}
It is straightforward.
\end{proof}

Let $V$ be a vector space. Denote the standard pairing between the dual space
$V^{\ast}$ and $V$ by
\begin{align*}
\langle-,-\rangle:\quad V^{\ast}\otimes V\rightarrow \Bbbk, \qquad\quad
\langle \xi,\; v \rangle:=\xi(v),
\end{align*}
for any $\xi\in V^{\ast}$ and $v\in V$. Let $V$, $W$ be two vector spaces. For a linear
map $\varphi: V\rightarrow W$, the transpose map $\varphi^{\ast}: W^{\ast}\rightarrow
V^{\ast}$ is defined by
\begin{align*}
\langle \varphi^{\ast}(\xi),\; v \rangle:=\langle\xi,\; \varphi(v)\rangle,
\end{align*}
for any $v\in V$ and $\xi\in W^{\ast}$. Let $(D, \dashv, \vdash)$ be a diassociative algebra
and $V$ be a vector space. For a linear map $\psi: D\rightarrow\gl(V)$, the linear map
$\psi^{\ast}: D\rightarrow\gl(V^{\ast})$ is defined by
\begin{align*}
\langle\psi^{\ast}(d)(\xi),\; v\rangle:=-\langle\xi,\; \psi(d)(v)\rangle,
\end{align*}
for any $d\in D$, $v\in V$, $\xi\in V^{\ast}$. That is, $\psi^{\ast}(d)=\psi(d)^{\ast}$
for all $d\in D$. Then, for any bimodule $(V, \kl_{\dashv}, \kr_{\dashv}, \kl_{\vdash},
\kr_{\vdash})$ over a diassociative algebra $(D, \dashv, \vdash)$, one can check that
$(V^{\ast}, \kr_{\vdash}^{\ast}-\kr_{\dashv}^{\ast}, -\kl_{\vdash}^{\ast},
-\kr_{\dashv}^{\ast}, \kl_{\dashv}^{\ast}-\kl_{\vdash}^{\ast})$ is again a
bimodule over $(D, \dashv, \vdash)$. In particular, we get $(D^{\ast}, \fr_{\vdash}^{\ast}
-\fr_{\dashv}^{\ast}, -\fl_{\vdash}^{\ast}, -\fr_{\dashv}^{\ast}, \fl_{\dashv}^{\ast}
-\fl_{\vdash}^{\ast})$ is a bimodule over $(D, \dashv, \vdash)$, which is called
the {\bf coregular bimodule}.

Let $V$ be a vector space and $\omega(-,-)$ be a bilinear form on $V$. Recall that
\begin{enumerate}\itemsep=0pt
\item[-] $\omega(-,-)$ is called {\bf nondegenerate} if $\omega(v_{1},\;
        v_{2})=0$ for any $v_{2}\in V$, then $v_{1}=0$;
\item[-] $\omega(-,-)$ is called {\bf symmetric} if $\omega(v_{1},\; v_{2})
        =\omega(v_{2},\; v_{1})$ for any $v_{1}, v_{2}\in V$;
\item[-] $\omega(-,-)$ is called {\bf skew-symmetric} if $\omega(v_{1},\; v_{2})
        =-\omega(v_{2},\; v_{1})$, for any $v_{1}, v_{2}\in V$.
\end{enumerate}
Let $\omega(-,-)$ be a bilinear form on a diassociative algebra $(D, \dashv, \vdash)$.
Then $\omega(-,-)$ is called {\bf invariant} if for any $d_{1}, d_{2}, d_{3}\in D$,
$$
\omega(d_{1}\vdash d_{2},\; d_{3})=\omega(d_{1},\; d_{2}\vdash d_{3}-d_{2}\dashv d_{3})
\qquad\mbox{and}\qquad \omega(d_{1}\dashv d_{2},\; d_{3})=\omega(d_{1},\; d_{2}\vdash d_{3}).
$$
A (skew-symmetric) {\bf quadratic diassociative algebra}
$(D, \dashv, \vdash, \omega)$ is a diassociative algebra $(D, \dashv, \vdash)$ with a
skew-symmetric, nondegenerate and invariant bilinear form $\omega(-,-): D\otimes
D\rightarrow\Bbbk$. By direct calculation, we have:

\begin{pro}\label{pro:dual}
Let $(D, \dashv, \vdash)$ be a diassociative algebra. Then the regular bimodule $(D,
\fl_{\dashv}, \fr_{\dashv}$, $\fl_{\vdash}, \fr_{\vdash})$ and the coregular bimodule
$(D^{\ast}, \fr_{\vdash}^{\ast}-\fr_{\dashv}^{\ast}, -\fl_{\vdash}^{\ast},
-\fr_{\dashv}^{\ast}, \fl_{\dashv}^{\ast}-\fl_{\vdash}^{\ast})$ are isomorphic as
bimodules over $(D, \dashv, \vdash)$ if there exists a skew-symmetric,
nondegenerate invariant bilinear form $\omega(-,-)$ on $(D, \dashv, \vdash)$.
\end{pro}

\subsection{Diassociative bialgebras}\label{subsec:dibiass}
Recall that a {\bf coassociative coalgebra} is pair $(C, \Delta)$ that $C$ is a
vector space and $\Delta: C\rightarrow C\otimes C$ satisfying the coassociativity
condition: $(\id\otimes\Delta)\circ\Delta=(\Delta\otimes\id)\circ\Delta$.

\begin{defi}[\cite{Lu,HLLZ}]\label{def:codi}
A {\bf diassociative coalgebra} is a triple $(C, \Delta_{\dashv}, \Delta_{\vdash})$
where $(C, \Delta_{\dashv})$ and $(C, \Delta_{\vdash})$ are coassociative coalgebras
satisfying the following conditions:
\begin{align}
(\Delta_{\dashv}\otimes\id)\circ\Delta_{\dashv}
&=(\id\otimes\Delta_{\vdash})\circ\Delta_{\dashv},    \label{codi1} \\
(\Delta_{\vdash}\otimes\id)\circ\Delta_{\dashv}
&=(\id\otimes\Delta_{\dashv})\circ\Delta_{\vdash},    \label{codi2}\\
(\Delta_{\dashv}\otimes\id)\circ\Delta_{\vdash}
&=(\Delta_{\vdash}\otimes\id)\circ\Delta_{\vdash}.    \label{codi3}
\end{align}
\end{defi}

One can check that $(C, \Delta_{\dashv}, \Delta_{\vdash})$ is a diassociative coalgebra
if and only if $(C^{\ast}, \Delta_{\dashv}^{\ast}, \Delta_{\vdash}^{\ast})$ is a
diassociative algebra.

\begin{defi}[\cite{Lu,HLLZ}]\label{def:bidi}
A \textbf{diassociative bialgebra} is a quintuple $(D, \dashv, \vdash, \Delta_{\dashv},
\Delta_{\vdash})$ where $(D, \dashv, \vdash)$ is a diassociative algebra and $(D,
\Delta_{\dashv}, \Delta_{\vdash})$ is a diassociative coalgebra satisfying the
following compatibility conditions:
\begin{align}
&\qquad\qquad (\id\otimes\fl_{\dashv}(d_{1}))(\Delta_{\dashv}(d_{2}))
=(\fr_{\vdash}(d_{2})\otimes\id)(\Delta_{\vdash}(d_{1})),         \label{bidi1}\\
&\qquad\qquad (\fl_{\dashv}(d_{1})\otimes\id)(\Delta_{\vdash}(d_{2}))
=(\id\otimes\fl_{\dashv}(d_{2}))(\tau(\Delta_{\vdash}(d_{1}))),    \label{bidi2}\\
&\qquad\qquad (\fr_{\vdash}(d_{1})\otimes\id)(\tau(\Delta_{\dashv}(d_{2})))
=(\id\otimes\fr_{\vdash}(d_{2}))(\Delta_{\dashv}(d_{1})),          \label{bidi3}\\
& \Delta_{\vdash}(d_{1}\vdash d_{2})
=(\id\otimes\fl_{\vdash}(d_{1}))(\Delta_{\vdash}(d_{2}))
-((\fr_{\vdash}-\fr_{\dashv})(d_{2})\otimes\id)
((\Delta_{\vdash}-\Delta_{\dashv})(d_{1})),                         \label{bidi4}\\
&\quad \Delta_{\vdash}(d_{1}\dashv d_{2})=(\id\otimes\fl_{\dashv}(d_{1}))
((\Delta_{\vdash}-\Delta_{\dashv})(d_{2}))
+(\fr_{\dashv}(d_{2})\otimes\id)(\Delta_{\vdash}(d_{1})),          \label{bidi5}\\
&\quad \Delta_{\dashv}(d_{1}\vdash d_{2})=(\id\otimes(\fl_{\vdash}-\fl_{\dashv})(d_{1}))
(\Delta_{\dashv}(d_{2}))+(\fr_{\vdash}(d_{2})\otimes\id)
(\Delta_{\dashv}(d_{1})),                                          \label{bidi6}\\
& \Delta_{\dashv}(d_{1}\dashv d_{2})=(\fr_{\dashv}(d_{2})\otimes\id)
(\Delta_{\dashv}(d_{1}))-(\id\otimes(\fl_{\vdash}-\fl_{\dashv})(d_{1}))
((\Delta_{\vdash}-\Delta_{\dashv})(d_{2})),                       \label{bidi7}\\
& ((\fl_{\vdash}-\fl_{\dashv})(d_{1})\otimes\id)(\Delta_{\vdash}(d_{2}))
+(\id\otimes\fl_{\dashv}(d_{2}))(\tau(\Delta_{\dashv}(d_{1})))    \label{bidi8}\\[-1mm]
&\qquad\qquad\quad=((\fr_{\vdash}-\fr_{\dashv})(d_{2})\otimes\id)
(\tau((\Delta_{\vdash}-\Delta_{\dashv})(d_{1})))
+(\id\otimes\fr_{\dashv}(d_{1}))(\Delta_{\vdash}(d_{2})),               \nonumber\\
& (\fr_{\dashv}(d_{1})\otimes\id)(\tau(\Delta_{\dashv}(d_{2})))+(\id\otimes
\fr_{\vdash}(d_{2}))((\Delta_{\vdash}-\Delta_{\dashv})(d_{1}))        \label{bidi9}\\[-1mm]
&\qquad\qquad\quad=((\fl_{\vdash}-\fl_{\dashv})(d_{2})\otimes\id)
((\Delta_{\vdash}-\Delta_{\dashv})(d_{1}))
+(\id\otimes\fl_{\vdash}(d_{1}))(\tau(\Delta_{\dashv}(d_{2}))),    \nonumber
\end{align}
for any $d_{1}, d_{2}\in D$.
\end{defi}

Let $(D, \dashv, \vdash)$ be a diassociative algebra. We define two linear maps
$\mathcal{E}, \mathcal{F}: D\rightarrow\gl(D\otimes D)$ by
\begin{align}
\mathcal{E}(d)&:=(\fr_{\vdash}-\fr_{\dashv})(d)\otimes\id
+\id\otimes\fl_{\dashv}(d),                  \label{cobdi1}\\
\mathcal{F}(d)&:=\id\otimes(\fl_{\vdash}-\fl_{\dashv})(d)
-\fr_{\vdash}(d)\otimes\id,                  \label{cobdi2}
\end{align}
for any $d\in D$. An element $r\in D\otimes D$ is called {\bf symmetric} if $r=\tau(r)$,
where $\tau: D\otimes D\rightarrow D\otimes D$ is the twist map defined by
$\tau(d_{1}\otimes d_{2}):=d_{2}\otimes d_{1}$ for all $d_{1}, d_{2}\in D$, and $r$
is called {\bf diass-invariant} if $\mathcal{E}(d)(r)=\mathcal{F}(d)(r)=0$ for any $d\in D$.
If there exists an element $r\in D\otimes D$ such that $(D, \dashv, \vdash, \Delta_{\dashv,r},
\Delta_{\vdash,r})$ is a diassociative bialgebra, where $\Delta_{\dashv,r}, \Delta_{\vdash,r}:
D\rightarrow D\otimes D$ are given by
\begin{align}
\Delta_{\dashv,r}(d)=\mathcal{F}(d)(r), \qquad\qquad
\Delta_{\vdash,r}(d)=-\mathcal{E}(d)(r),    \label{cobdi}
\end{align}
for any $d\in D$, then $(D, \dashv, \vdash, \Delta_{\dashv,r}, \Delta_{\vdash,r})$ is called a
{\bf coboundary diassociative bialgebra} associated with $r$. Let $(D, \dashv, \vdash)$
be a diassociative algebra and $r=\sum_{i}x_{i}\otimes y_{i}\in D\otimes D$. The equation
$$
\mathbf{D}_{r}=r_{13}\vdash r_{23}-r_{23}\dashv r_{12}
-r_{12}\vdash r_{13}+r_{12}\dashv r_{13}=0
$$
is called the {\bf diassociative Yang-Baxter equation} (or $\DAYBE$) in the diassociative
algebra $(D, \dashv, \vdash)$, where $r_{13}\vdash r_{23}=\sum_{i,j}x_{i}\otimes x_{j}
\otimes(y_{i}\vdash y_{j})$, $r_{23}\dashv r_{12}=\sum_{i,j}x_{j}\otimes(x_{i}\dashv
y_{j})\otimes y_{i}$, $r_{12}\vdash r_{13}=\sum_{i,j}(x_{i}\vdash x_{j})\otimes y_{i}
\otimes y_{j}$ and $r_{12}\dashv r_{13}=\sum_{i,j}(x_{i}\dashv x_{j})\otimes y_{i}
\otimes y_{j}$.

\begin{pro}[\cite{Lu,HLLZ}]\label{pro:quasi-di}
Let $(D, \dashv, \vdash)$ be a diassociative algebra, $r\in D\otimes D$ and
$\Delta_{\dashv,r}, \Delta_{\vdash,r}: D\rightarrow D\otimes D$ be the linear maps
defined by Eq. \eqref{cobdi}.
\begin{enumerate}
\item[$(i)$] If $r$ is a solution of the $\DAYBE$ in $(D, \dashv, \vdash)$ and $r-\tau(r)$ is
     diass-invariant, then $(D, \dashv, \vdash$, $\Delta_{\dashv,r}, \Delta_{\vdash,r})$
     is a diassociative bialgebra, which is called a {\bf quasi-triangular diassociative
     bialgebra} associated with $r$.
\item[$(ii)$] If $r$ is a symmetric solution of the $\DAYBE$ in $(D, \dashv, \vdash)$, then
     $(D, \dashv, \vdash, \Delta_{\dashv,r}, \Delta_{\vdash,r})$ is a Leibniz bialgebra,
     which is called a {\bf triangular diassociative bialgebra} associated with $r$.
\end{enumerate}
\end{pro}

Let $V$ be a vector space. For any $r\in V\otimes V$,
we define a linear map $r^{\sharp}: V^{\ast}\rightarrow V$ by
$$
\langle r^{\sharp}(\xi_{1}),\; \xi_{2}\rangle=\langle\xi_{1}\otimes\xi_{2},\; r\rangle,
$$
for any $\xi_{1}, \xi_{2}\in V^{\ast}$.

\begin{defi}[\cite{Lu}]\label{def:fact-di}
Let $(D, \dashv, \vdash)$ be a diassociative algebra, $r\in D\otimes D$ and
$(D, \dashv, \vdash, \Delta_{\dashv,r}, \Delta_{\vdash,r})$ be a quasi-triangular
diassociative bialgebra associated with $r$. If $\mathcal{I}=r^{\sharp}-\tau(r)^{\sharp}:
D^{\ast}\rightarrow D$ is an isomorphism of vector spaces,
then $(D, \dashv, \vdash, \Delta_{\dashv,r}, \Delta_{\vdash,r})$ is called a
{\bf factorizable diassociative bialgebra}.
\end{defi}

In a factorizable diassociative bialgebra $(D, \dashv, \vdash, \Delta_{\dashv,r},
\Delta_{\vdash,r})$, each element $d\in D$ can be decomposed into $d=d_{+}+d_{-}$,
where $d_{+}\in\Img(r^{\sharp})$ and $d_{-}\in\Img(\tau(r)^{\sharp})$ \cite{Lu}.
Note that $\mathcal{I}=0$ if $(D, \dashv, \vdash, \Delta_{\dashv,r}, \Delta_{\vdash,r})$
is a triangular diassociative bialgebra, we can view the factorizable diassociative
bialgebra is the opposite of the triangular diassociative bialgebra.
In this paper, we consider the diassociative bialgebras induced by ASI bialgebras,
and discuss its connection with Lie bialgebras and Leibniz bialgebras.

\section{Quasi-triangular diassociative bialgebras from antisymmetric infinitesimal
bialgebras} \label{sec:qtbia-ASI}
In this section, we recall the notions of antisymmetric infinitesimal bialgebras and
quadratic perm algebras. We show that there is a diassociative bialgebra on the tensor
product of an antisymmetric infinitesimal bialgebra and a quadratic perm algebra,
and prove that the induced diassociative bialgebra is coboundary (resp. quasi-triangular,
triangular, factorizable) if the antisymmetric infinitesimal bialgebra is
coboundary (resp. quasi-triangular, triangular, factorizable).

Let $(P, \diamond)$ be a perm algebra. Recall that a {\bf bimodule $(V, \kkl, \kkr)$ over
$(P, \diamond)$} is a vector space $V$ with two linear maps $\kkl, \kkr: P\rightarrow\gl(V)$
such that for any $p_{1}, p_{2}\in P$,
\begin{align*}
&\qquad\quad\kkl(p_{1}\diamond p_{2})=\kkl(p_{1})\circ\kkl(p_{2})
=\kkl(p_{2})\circ\kkl(p_{1}),\\
&\kkr(p_{1}\diamond p_{2})=\kkr(p_{2})\circ\kkr(p_{1})=\kkr(p_{2})\circ\kkl(p_{1})
=\kkl(p_{1})\circ\kkr(p_{2}).
\end{align*}
In particular, $(P, \ffl_{P}, \ffr_{P})$ is a bimodule over $(P, \diamond)$, which is
called the regular bimodule, where $\ffl_{P}, \ffr_{P}: P\rightarrow\gl(P)$ are
given by $\ffl_{P}(p_{1})(p_{2})=p_{1}\diamond p_{2}=\ffr_{P}(p_{2})(p_{1})$ for any
$p_{1}, p_{2}\in P$. We now consider dual conclusion for Proposition \ref{pro:ass-diass}.

\begin{defi}\label{def:permcoalg}
A {\bf perm coalgebra} $(P, \nu)$ is a vector space $P$
with a linear map $\nu: P\rightarrow P\otimes P$ satisfying
$$
(\nu\otimes\id)\circ\nu=(\id\otimes\nu)\circ\nu=(\tau\otimes\id)\circ(\id\otimes\nu)\circ\nu.
$$
\end{defi}

One can check that $(P, \nu)$ is a perm coalgebra if and only if
$(P^{\ast}, \nu^{\ast})$ is a perm algebra.
For the dual version of Proposition \ref{pro:ass-diass}, we have

\begin{pro}\label{pro:perm-codia}
Let $(C, \Delta)$ be a coassociative coalgebra and $(P, \nu)$ be a perm coalgebra.
Define two linear maps $\Delta_{\dashv}, \Delta_{\vdash}: C\otimes P\rightarrow
(C\otimes P)\otimes(C\otimes P)$ by
\begin{align*}
\Delta_{\dashv}(c\otimes p)&=\Delta(c)\bullet\tau(\nu(p))=\sum_{(c)}\sum_{(p)}
(c_{(1)}\otimes p_{(2)})\otimes(c_{(2)}\otimes p_{(1)}),\\[-2mm]
\Delta_{\vdash}(c\otimes p)&=\Delta(c)\bullet\nu(p)=\sum_{(c)}\sum_{(p)}
(c_{(1)}\otimes p_{(1)})\otimes(c_{(2)}\otimes p_{(2)}),
\end{align*}
for any $c\in C$ and $p\in P$, where $\Delta(c)=\sum_{(c)}c_{(1)}\otimes c_{(2)}$,
and $\nu(p)=\sum_{(p)}p_{(1)}\otimes p_{(2)}$ in the Sweedler notation.
Then $(C\otimes P, \Delta_{\dashv}, \Delta_{\vdash})$ is a diassociative coalgebra.
\end{pro}

\begin{proof}
For any $\sum_{l}c'_{l}\otimes c''_{l}\otimes c'''_{l}\in C\otimes C\otimes C$ and
$\sum_{k}p'_{k}\otimes p''_{k}\otimes p'''_{k}\in P\otimes P\otimes P$, we denote
$$
\Big(\sum_{l}c'_{l}\otimes c''_{l}\otimes c'''_{l}\Big)\bullet
\Big(\sum_{k}p'_{k}\otimes p''_{k}\otimes p'''_{k}\Big)
=\sum_{l}\sum_{k}(c'_{l}\otimes p'_{k})\otimes
(c''_{l}\otimes p''_{k})\otimes(c'''_{l}\otimes p'''_{k}).
$$
Then, by using the above notations, since $(P, \nu)$ is a perm coalgebra
and $(C, \Delta)$ is a coassociative coalgebra, for any $c\otimes p\in C\otimes P$, we have
\begin{align*}
(\Delta_{\dashv}\otimes\id)(\Delta_{\dashv}(c\otimes p))
=&\;(\Delta\otimes\id)(\Delta(c))\bullet
(\tau\otimes\id)((\id\otimes\tau)((\tau\otimes\id)((\id\otimes\nu)(\nu(p)))))\\
=&\;(\id\otimes\Delta)(\Delta(c))\bullet
(\tau\otimes\id)((\id\otimes\tau)((\nu\otimes\id)(\nu(p))))\\
=&\;(\id\otimes\Delta_{\vdash})(\Delta_{\dashv}(c\otimes p)).
\end{align*}
That is, $(\Delta_{\dashv}\otimes\id)\circ\Delta_{\dashv}=(\id\otimes\Delta_{\vdash})
\circ\Delta_{\dashv}$. Similarly, we also have $(\Delta_{\vdash}\otimes\id)\circ\Delta_{\dashv}
=(\id\otimes\Delta_{\dashv})\circ\Delta_{\vdash}$ and $(\Delta_{\dashv}\otimes\id)
\circ\Delta_{\vdash}=(\Delta_{\vdash}\otimes\id)\circ\Delta_{\vdash}$.
Thus, $(C\otimes P, \Delta_{\dashv}, \Delta_{\vdash})$ is a diassociative coalgebra.
\end{proof}
Let $(P, \diamond)$ be a perm algebra. A bilinear form $\varpi(-,-)$ on $(P, \diamond)$
is called {\bf invariant} if it satisfies
$$
\varpi(p_{1}\diamond p_{2},\; p_{3})
=\varpi(p_{1},\; p_{2}\diamond p_{3}-p_{3}\diamond p_{2}),
$$
for any $p_{1}, p_{2}, p_{3}\in P$. A {\bf quadratic perm algebra}, denoted by
$(P, \diamond, \varpi)$, is a perm algebra $(P, \diamond)$ together with a skew-symmetric
invariant nondegenerate bilinear form $\varpi(-,-)$. Let $(P, \diamond, \varpi)$ be a
quadratic perm algebra. Then the bilinear form $\varpi(-,-)$ can naturally expand to
the tensor product $P\otimes P\otimes\cdots\otimes P$, i.e.,
$$
\varpi(-,-):\qquad (\underbrace{P\otimes\cdots\otimes P}_{\mbox{\tiny $k$-fold}})\otimes
(\underbrace{P\otimes\cdots\otimes P}_{\mbox{\tiny $k$-fold}})\longrightarrow\Bbbk,
$$
$\varpi(p_{1}\otimes p_{2}\otimes\cdots\otimes p_{k},\ \ p'_{1}\otimes p'_{2}
\otimes\cdots\otimes p'_{k})=\prod_{i=1}^{k}\varpi(p_{i}, p'_{i})$,
for any $p_{1}, p_{2},\cdots, p_{k}, p'_{1}, p'_{2},\cdots, p'_{k}\in P$.
Then $\varpi(-,-)$ on $P\otimes P\otimes\cdots\otimes P$ is also a
nondegenerate bilinear form.

\begin{lem}[\cite{LZB}]\label{lem:perm-dual}
Let $(P, \diamond, \varpi)$ be a quadratic perm algebra. Define a linear map
$\nu_{\varpi}: P\rightarrow P\otimes P$ by $\varpi(\nu_{\varpi}(p_{1}),\;
p_{2}\otimes p_{3})=\varpi(p_{1},\; p_{2}\diamond p_{3})$ for any $p_{1}, p_{2},
p_{3}\in P$. Then $(P, \nu_{\varpi})$ is a perm coalgebra.
\end{lem}

\begin{ex}\label{ex:dual-perm}
Consider the $2$-dimensional perm algebra $(P, \diamond)$, where $P={\rm span}_{\Bbbk}
\{x_{1}, x_{2}\}$, the products are given by $x_{2}\diamond x_{1}=x_{1}$, $x_{2}\diamond x_{2}
=x_{2}$, $x_{1}\diamond x_{1}=0=x_{1}\diamond x_{2}$. If we define a skew-symmetric bilinear
form $\varpi(-,-)$ on $P$ by $\varpi(x_{1}, x_{2})=1$. Then $(P, \diamond, \varpi)$ is a
quadratic perm algebra, and we get a perm coalgebra $(P, \nu_{\omega})$, where
$\nu_{\omega}(x_{1})=x_{1}\otimes x_{1}$ and $\nu_{\omega}(x_{2})=x_{1}\otimes x_{2}$.
\end{ex}

%
Recall that an {\bf antisymmetric infinitesimal bialgebra} (ASI bialgebra) is a triple
$(A, \cdot, \Delta)$ such that $(A, \cdot)$ is an associative algebra, $(A, \Delta)$
is a coassociative coalgebra, and the following compatibility condition holds:
\begin{align}
&\qquad\quad\Delta(a_{1}\cdot a_{2})=(\fr_{A}(a_{2})\otimes\id)(\Delta(a_{1}))
+(\id\otimes\fl_{A}(a_{1}))(\Delta(a_{2})), \label{ASI1}\\
&\big(\fl_{A}(a_{1})\otimes\id-\id\otimes\fr_{A}(a_{1})\big)(\Delta(a_{2}))
=\tau\big(\big(\id\otimes\fr_{A}(a_{2})
-\fl_{A}(a_{2})\otimes\id\big)(\Delta(a_{1}))\big), \label{ASI2}
\end{align}
for any $a_{1}, a_{2}\in A$. Following from Propositions \ref{pro:ass-diass},
\ref{pro:perm-codia} and Lemma \ref{lem:perm-dual}, we can construct a diassociative
bialgebra from an ASI bialgebra and a quadratic perm algebra as the following theorem.

\begin{thm}\label{thm:permbia-dia}
Let $(A, \cdot, \Delta)$ be an ASI bialgebra, $(P, \diamond, \varpi)$ be a quadratic
perm algebra and $(A\otimes P, \dashv, \vdash)$ be the induced diassociative algebra by
$(A, \cdot)$ and $(P, \diamond)$ in Propositions \ref{pro:ass-diass}. Define two linear maps
$\Delta_{\dashv}, \Delta_{\vdash}: A\otimes P\rightarrow(A\otimes P)\otimes(A\otimes P)$ by
\begin{align}
\Delta_{\dashv}(a\otimes p)&=\Delta(a)\bullet\tau(\nu_{\varpi}(p))=\sum_{(a)}\sum_{(p)}
(a_{(1)}\otimes p_{(2)})\otimes(a_{(2)}\otimes p_{(1)}), \label{codia1}\\[-2mm]
\Delta_{\vdash}(a\otimes p)&=\Delta(a)\bullet\nu_{\varpi}(p)=\sum_{(a)}\sum_{(p)}
(a_{(1)}\otimes p_{(1)})\otimes(a_{(2)}\otimes p_{(2)}),  \label{codia2}
\end{align}
for any $a\in A$ and $p\in P$, where $\nu_{\varpi}(p)=\sum_{(p)}p_{(1)}\otimes p_{(2)}$
and $\Delta(a)=\sum_{(a)}a_{(1)}\otimes a_{(2)}$ in the Sweedler notation.
Then $(A\otimes P, \dashv, \vdash, \Delta_{\dashv}, \Delta_{\vdash})$ is a diassociative
bialgebra, which is called {\bf the diassociative bialgebra induced from
$(A, \cdot, \Delta)$ by $(P, \diamond, \varpi)$}.
\end{thm}

\begin{proof}
By Proposition \ref{pro:perm-codia} and Lemma \ref{lem:perm-dual}, we get
that $(A\otimes P, \Delta_{\dashv}, \Delta_{\vdash})$ is a diassociative coalgebra.
Thus, we only need to show Eqs. \eqref{bidi1}-\eqref{bidi9} hold. For any $e, f, p, p'\in P$,
note that
\begin{align*}
&\varpi((\id\otimes\ffr_{P}(p))(\tau(\nu_{\varpi}(p'))),\; e\otimes f)
=\varpi(p',\; (p\diamond f-f\diamond p)\diamond e)=0,\\
&\varpi((\ffr_{P}(p')\otimes\id)(\nu_{\varpi}(p)),\; e\otimes f)
=\varpi(p,\; (p'\diamond e-e\diamond p')\diamond f)=0.
\end{align*}
We get $(\id\otimes\ffr_{P}(p))(\tau(\nu_{\varpi}(p')))=0=(\ffr_{P}(p')\otimes\id)
(\nu_{\varpi}(p))$. Thus $(\id\otimes\fl_{\dashv}(a\otimes p))(\Delta_{\dashv}(a'\otimes p'))
=(\id\otimes\fl_{A}(a))(\Delta(a'))\bullet(\id\otimes\ffr_{P}(p))(\tau(\nu_{\varpi}(p')))=0$
and $(\fr_{\vdash}(a'\otimes p')\otimes\id)(\Delta_{\vdash}(a\otimes p))=
(\fr_{A}(a')\otimes\id)(\Delta(a))\bullet(\ffr_{P}(p')\otimes\id)(\nu_{\varpi}(p))=0$
That is, Eq. \eqref{bidi1} holds. Similarly, Eqs. \eqref{bidi2} and \eqref{bidi3} hold.
Moreover, for any $e, f, p, p'\in P$, since
\begin{align*}
&\varpi(\nu_{\varpi}(p'\diamond p),\; e\otimes f)
=\varpi\big(p',\; p\diamond(e\diamond f)-(e\diamond f)\diamond p\big),\\
&\varpi((\id\otimes\ffr_{P}(p))(\nu_{\varpi}(p')),\; e\otimes f)
=\varpi\big(p',\; e\diamond(p\diamond f)-e\diamond(f\diamond p)\big),\\
&\varpi(\tau((\ffr_{P}(p)\otimes\id)(\nu_{\varpi}(p'))),\; e\otimes f)
=\varpi\big(p',\; (p\diamond f)\diamond e-(f\diamond p)\diamond e\big)=0,\\
&\varpi((\ffl_{P}(p')\otimes\id)(\nu_{\varpi}(p)),\; e\otimes f)
=\varpi\big(p',\; p\diamond(e\diamond f)-(e\diamond f)\diamond p\big),
\end{align*}
we obtain that $\tau((\ffr_{P}(p)\otimes\id)(\nu_{\varpi}(p')))=0$ and
$\nu_{\varpi}(p'\diamond p)=(\id\otimes\ffr_{P}(p))(\nu_{\varpi}(p'))
=(\ffl_{P}(p')\otimes\id)(\nu_{\varpi}(p))$. Thus, by Eq. \eqref{ASI1}, we have
\begin{align*}
&\;\Delta_{\vdash}((a\otimes p)\dashv(a'\otimes p'))
-(\id\otimes\fl_{\dashv}(a\otimes p))(\Delta_{\vdash}(a'\otimes p'))\\[-1mm]
&\qquad +(\id\otimes\fl_{\dashv}(a\otimes p))(\Delta_{\dashv}(a'\otimes p'))
-(\fr_{\dashv}(a'\otimes p')\otimes\id)(\Delta_{\vdash}(a\otimes p))\\
=&\; \Delta(a\cdot a')\bullet\nu_{\varpi}(p'\diamond p)
-(\id\otimes\fl_{A}(a))(\Delta(a'))\bullet(\id\otimes\ffr_{P}(p))(\nu_{\varpi}(p'))\\[-1mm]
&\qquad +(\id\otimes\fl_{A}(a))(\Delta(a'))\bullet\tau((\ffr_{P}(p)\otimes\id)
(\nu_{\varpi}(p')))\\[-1mm]
&\qquad -(\fr_{A}(a')\otimes\id)(\Delta(a))\bullet(\ffl_{P}(p')\otimes\id)(\nu_{\varpi}(p))\\
=&\; \Big(\Delta(a\cdot a')-(\id\otimes\fl_{A}(a))(\Delta(a'))
-(\fr_{A}(a')\otimes\id)(\Delta(a))\Big)\bullet\nu_{\varpi}(p'\diamond p)\\
=&\; 0.
\end{align*}
That is, Eq. \eqref{bidi5} holds. Similarly, by Eqs. \eqref{ASI1} and \eqref{ASI2},
we get that Eqs. \eqref{bidi4}, \eqref{bidi6}-\eqref{bidi9} hold.
Thus, $(A\otimes P, \dashv, \vdash, \Delta_{\dashv}, \Delta_{\vdash})$ is a
diassociative bialgebra.
\end{proof}

\begin{ex}\label{ex:ind-diabi}
We consider $2$-dimensional ASI bialgebra $(A, \cdot, \Delta)$, where vector space
$A={\rm span}_{\Bbbk}\{e_{1}$, $e_{2}\}$ and the products and coporducts are given by
$e_{1}\cdot e_{1}=e_{1}$, $e_{1}\cdot e_{2}=e_{2}$, $\Delta(e_{1})=e_{2}\otimes e_{1}$
and $\Delta(e_{2})=e_{2}\otimes e_{2}$. Let $(P, \diamond, \varpi)$ be the $2$-dimensional
quadratic perm algebra given in Example \ref{ex:dual-perm}. By Theorem \ref{thm:permbia-dia},
we get a $4$-dimensional diassociative bialgebra $(A\otimes P, \dashv, \vdash,
\Delta_{\dashv}, \Delta_{\vdash})$, where
\begin{align*}
& (e_{1}\otimes x_{2})\vdash(e_{1}\otimes x_{1})=e_{1}\otimes x_{1},&
& (e_{1}\otimes x_{2})\vdash(e_{2}\otimes x_{1})=e_{2}\otimes x_{1},\\
& (e_{1}\otimes x_{2})\vdash(e_{1}\otimes x_{2})=e_{1}\otimes x_{2},&
& (e_{1}\otimes x_{2})\vdash(e_{2}\otimes x_{2})=e_{2}\otimes x_{2},\\
& (e_{1}\otimes x_{1})\dashv(e_{1}\otimes x_{2})=e_{1}\otimes x_{1},&
& (e_{1}\otimes x_{1})\dashv(e_{2}\otimes x_{2})=e_{2}\otimes x_{1},\\
& (e_{1}\otimes x_{2})\dashv(e_{1}\otimes x_{2})=e_{1}\otimes x_{2},&
& (e_{1}\otimes x_{2})\dashv(e_{2}\otimes x_{2})=e_{2}\otimes x_{2},\\
& \Delta_{\vdash}(e_{1}\otimes x_{1})=(e_{2}\otimes x_{1})\otimes(e_{1}\otimes x_{1})
=\Delta_{\dashv}(e_{1}\otimes x_{1}),&
& \Delta_{\vdash}(e_{1}\otimes x_{2})=(e_{2}\otimes x_{1})\otimes(e_{1}\otimes x_{2}),\\
& \Delta_{\vdash}(e_{2}\otimes x_{1})=(e_{2}\otimes x_{1})\otimes(e_{2}\otimes x_{1})
=\Delta_{\dashv}(e_{2}\otimes x_{1}),&
& \Delta_{\vdash}(e_{2}\otimes x_{2})=(e_{2}\otimes x_{1})\otimes(e_{2}\otimes x_{2}),\\
& \Delta_{\dashv}(e_{1}\otimes x_{2})=(e_{2}\otimes x_{2})\otimes(e_{1}\otimes x_{1}), &
& \Delta_{\dashv}(e_{2}\otimes x_{2})=(e_{2}\otimes x_{2})\otimes(e_{2}\otimes x_{1}).
\end{align*}
\end{ex}

Recall that an ASI bialgebra $(A, \cdot, \Delta)$ is called {\bf coboundary}
if there exists an element $r\in A\otimes A$ such that $\Delta=\Delta_{r}$, where
\begin{align}
\Delta_{r}(a)=(\id\otimes\fl_{A}(a)-\fr_{A}(a)\otimes\id)(r), \label{ass-cobo}
\end{align}
for any $a\in A$. Let $(A, \cdot)$ be an associative algebra. An element
$r=\sum_{i}x_{i}\otimes y_{i}\in A\otimes A$ is said to be {\bf skew-symmetric}
if $r=-\tau(r)$, and it is said to be {\bf ass-invariant}
if $\Delta_{r}(a)=0$ for all $a\in A$. The equation
$$
\mathbf{A}_{r}:=r_{12}\cdot r_{13}+r_{13}\cdot r_{23}-r_{23}\cdot r_{12}=0
$$
is called the {\bf associative Yang-Baxter equation} (or $\AYBE$) in $(A, \cdot)$,
where $r_{12}\cdot r_{13}=\sum_{i,j}(x_{i}\cdot x_{j})\otimes y_{i}\otimes y_{j}$,
$r_{13}\cdot r_{23}=\sum_{i,j}x_{i}\otimes x_{j}\otimes(y_{i}\cdot y_{j})$ and
$r_{23}\cdot r_{12}=\sum_{i,j}x_{j}\otimes(x_{i}\cdot y_{j})\otimes y_{i}$.

\begin{pro}[\cite{SW}]\label{pro:quasass-bia}
Let $(A, \cdot)$ be an associative algebra, $r\in A\otimes A$ and $\Delta_{r}:
A\rightarrow A\otimes A$ be the linear map defined by Eq. \eqref{ass-cobo}.
\begin{enumerate}\itemsep=0pt
\item[$(i)$] If $r$ is a skew-symmetric solution of the $\AYBE$ in $(A, \cdot)$, then
     $(A, \cdot, \Delta_{r})$ is an ASI bialgebra, which is called a {\bf triangular ASI
     bialgebra} associated with $r$.
\item[$(ii)$] If $r$ is a solution of the $\AYBE$ in $(A, \cdot)$ and $r+\tau(r)$ is
     ass-invariant, then $(A, \cdot, \Delta_{r})$ is an ASI bialgebra,
     which is called a {\bf quasi-triangular ASI bialgebra} associated with $r$.
\end{enumerate}
\end{pro}
%
%
A quasi-triangular ASI bialgebra $(A, \cdot, \Delta_{r})$ is called a {\bf factorizable
ASI bialgebra} if $\mathcal{I}=r^{\sharp}+\tau(r)^{\sharp}: A^{\ast}\rightarrow A$
is an isomorphism of vector spaces.
These special ASI bialgebras are all related to the solutions of the associative
Yang-Baxter equation. Next, we consider the relation between the solutions of the
$\AYBE$ in an associative algebra and the solutions of the $\DAYBE$ in the induced
diassociative algebra. Let $(P, \diamond, \varpi)$ be a quadratic perm algebra
and $\{e_{1}, e_{2},\cdots, e_{n}\}$ be a basis of $P$. Since $\varpi(-,-)$ is
skew-symmetric nondegenerate, we get a basis $\{f_{1}, f_{2},\cdots, f_{n}\}$ of $P$,
which is called the dual basis of $\{e_{1}, e_{2},\cdots, e_{n}\}$ with respect to
$\varpi(-,-)$, by $\varpi(f_{i}, e_{j})=\delta_{ij}$, where $\delta_{ij}$ is the Kronecker
delta. Here, for some special solutions of the $\AYBE$ in an associative algebra, we have:

\begin{pro}\label{pro:AYBE-DYBE}
Let $(A, \cdot)$ be an associative algebra, $(P, \diamond, \varpi)$ be a quadratic
perm algebra and $(A\otimes P, \dashv, \vdash)$ be the induced diassociative algebra.
Suppose that $r=\sum_{i}x_{i}\otimes y_{i}\in A\otimes A$ is a solution of the $\AYBE$
in $(A, \cdot)$, $\{e_{1}, e_{2},\cdots, e_{n}\}$ is a basis of $P$ and $\{f_{1},
f_{2},\cdots, f_{n}\}$ is the dual basis of $\{e_{1}, e_{2},\cdots, e_{n}\}$ with
respect to $\varpi(-,-)$. Then
\begin{align}
\widetilde{r}=\sum_{i, j}(x_{i}\otimes e_{j})\otimes(y_{i}\otimes f_{j})
\in(A\otimes P)\otimes(A\otimes P)  \label{r-max}
\end{align}
is a solution of the $\DAYBE$ in $(A\otimes P, \dashv, \vdash)$.
Moreover, we get that $\widetilde{r}-\tau(\widetilde{r})$ is diass-invariant
if $r+\tau(r)$ is ass-invariant.

In particular, $\widetilde{r}$ is a symmetric solution of the $\DAYBE$ in
$(A\otimes P, \dashv, \vdash)$ if $r$ is a skew-symmetric solution of the
$\AYBE$ in $(A, \cdot)$.
\end{pro}

\begin{proof}
First, by direct calculation, we have
\begin{align*}
&\; \widetilde{r}_{13}\vdash\widetilde{r}_{23}-\widetilde{r}_{23}\dashv\widetilde{r}_{12}
-\widetilde{r}_{12}\vdash\widetilde{r}_{13}+\widetilde{r}_{12}\dashv\widetilde{r}_{13}\\
=&\; \sum_{i,j}\sum_{k,l}\Big(\big(x_{i}\otimes x_{j}\otimes(y_{i}\cdot y_{j})\big)
\bullet\big(e_{k}\otimes e_{l}\otimes(f_{k}\diamond f_{l})\big)
-\big(x_{j}\otimes(x_{i}\cdot y_{j})\otimes y_{i}\big)\bullet
\big(e_{l}\otimes(f_{l}\diamond e_{k})\otimes f_{k}\big)\\[-4mm]
&\qquad\quad-\big((x_{i}\cdot x_{j})\otimes y_{i}\otimes y_{j}\big)\bullet
\big((e_{k}\diamond e_{l})\otimes f_{k}\otimes f_{l}\big)
+\big((x_{i}\cdot x_{j})\otimes y_{i}\otimes y_{j}\big)\bullet
\big((e_{l}\diamond e_{k})\otimes f_{k}\otimes f_{l}\big)\Big).
\end{align*}
Moreover, since $\varpi(-,-)$ is nondegenerate and for any $1\leq s, u, v\leq n$,
\begin{align*}
\varpi\Big(\sum_{k,l}e_{k}\otimes e_{l}\otimes(f_{k}\diamond f_{l}),\ \
e_{s}\otimes e_{u}\otimes e_{v}\Big)&=\varpi(e_{v}\diamond e_{u}
+e_{u}\diamond e_{v},\; e_{s}),\\[-2mm]
\varpi\Big(\sum_{k,l}e_{l}\otimes(f_{l}\diamond e_{k})\otimes f_{k},\ \
e_{s}\otimes e_{u}\otimes e_{v}\Big)&=\varpi(e_{v}\diamond e_{u}
+e_{u}\diamond e_{v},\; e_{s}),\\[-2mm]
\varpi\Big(\sum_{k,l}(e_{k}\diamond e_{l})\otimes f_{k}\otimes f_{l},\ \
e_{s}\otimes e_{u}\otimes e_{v}\Big)&=\varpi(e_{u}\diamond e_{v},\; e_{s}),\\[-2mm]
\varpi\Big(\sum_{k,l}(e_{l}\diamond e_{k})\otimes f_{k}\otimes f_{l},\ \
e_{s}\otimes e_{u}\otimes e_{v}\Big)&=\varpi(e_{v}\diamond e_{u},\; e_{s}),
\end{align*}
we get that $\sum_{k,l}e_{k}\otimes e_{l}\otimes(f_{k}\diamond f_{l})=
\sum_{k,l}e_{l}\otimes(f_{l}\diamond e_{k})\otimes f_{k}=\sum_{k,l}(e_{k}\diamond e_{l})
\otimes f_{k}\otimes f_{l}+\sum_{k,l}(e_{l}\diamond e_{k})\otimes f_{k}\otimes f_{l}$,
and so that
\begin{align*}
&\; \widetilde{r}_{13}\vdash\widetilde{r}_{23}-\widetilde{r}_{23}\dashv\widetilde{r}_{12}
-\widetilde{r}_{12}\vdash\widetilde{r}_{13}+\widetilde{r}_{12}\dashv\widetilde{r}_{13}\\
=&\; \sum_{i,j}\sum_{k,l}\Big(\big((x_{i}\cdot x_{j})\otimes y_{i}\otimes y_{j}
+x_{i}\otimes x_{j}\otimes(y_{i}\cdot y_{j})-x_{j}\otimes(x_{i}\cdot y_{j})\otimes y_{i}\big)
\bullet\big((e_{l}\diamond e_{k})\otimes f_{k}\otimes f_{l}\big)\\[-4mm]
&\qquad\quad-\big((x_{i}\cdot x_{j})\otimes y_{i}\otimes y_{j}
+x_{i}\otimes x_{j}\otimes(y_{i}\cdot y_{j})-x_{j}\otimes(x_{i}\cdot y_{j})\otimes y_{i}\big)
\bullet\big((e_{k}\diamond e_{l})\otimes f_{k}\otimes f_{l}\big)\\
=&\;\mathbf{A}_{r}\bullet\sum_{k,l}\big((e_{l}\diamond e_{k})\otimes f_{k}\otimes f_{l}
-(e_{k}\diamond e_{l})\otimes f_{k}\otimes f_{l}\big).
\end{align*}
That is to say, $\widetilde{r}$ is a solution of the $\DAYBE$ in $(A\otimes P, \dashv,
\vdash)$ if $r$ is a solution of the $\AYBE$ in $(A, \cdot)$.

Second, for any $a\in A$ and $p\in P$, we have
\begin{align*}
&\;\big((\fr_{\vdash}-\fr_{\dashv})(a\otimes p)\otimes\id
+\id\otimes\fl_{\dashv}(a\otimes p)\big)\big(\widetilde{r}-\tau(\widetilde{r})\big)\\
=&\;\sum_{i, j}\Big(((x_{i}\cdot a)\otimes y_{i})\bullet((e_{j}\diamond p)\otimes f_{j})
-((x_{i}\cdot a)\otimes y_{i})\bullet((p\diamond e_{j})\otimes f_{j})\\[-5mm]
&\qquad\quad +(x_{i}\otimes(a\cdot y_{i}))\bullet(e_{j}\otimes(f_{j}\diamond p))
-((y_{i}\cdot a)\otimes x_{i})\bullet((f_{j}\diamond p)\otimes e_{j})\\[-2mm]
&\qquad\quad +((y_{i}\cdot a)\otimes x_{i})\bullet((p\diamond f_{j})\otimes e_{j})
-(y_{i}\otimes(a\cdot x_{i}))\bullet(f_{j}\otimes(e_{j}\diamond p))\Big).
\end{align*}
Moreover, for any $1\leq s, t\leq n$, since
$$
\varpi\Big(\sum_{j}(e_{j}\diamond p)\otimes f_{j},\; e_{s}\otimes e_{t}\Big)
=\varpi(p,\; e_{t}\diamond e_{s})
=-\varpi\Big(\sum_{j}(f_{j}\diamond p)\otimes e_{j},\; e_{s}\otimes e_{t}\Big),
$$
we get $\sum_{j}(e_{j}\diamond p)\otimes f_{j}=-\sum_{j}(f_{j}\diamond p)\otimes e_{j}$.
Similarly, we also have $\sum_{j}e_{j}\otimes(f_{j}\diamond p)
=-\sum_{j}f_{j}\otimes(e_{j}\diamond p)$ and $\sum_{j}(p\diamond e_{j})\otimes f_{j}=
-\sum_{j}(p\diamond f_{j})\otimes e_{j}=\sum_{j}(e_{j}\diamond p)\otimes f_{j}
+e_{j}\otimes(f_{j}\diamond p)$. Therefore, we obtain
\begin{align*}
&\;\big((\fr_{\vdash}-\fr_{\dashv})(a\otimes p)\otimes\id
+\id\otimes\fl_{\dashv}(a\otimes p)\big)\big(\widetilde{r}-\tau(\widetilde{r})\big)\\
=&\;\sum_{i,j}\big((x_{i}\cdot a)\otimes y_{i}+(y_{i}\cdot a)\otimes x_{i}
-x_{i}\otimes(a\cdot y_{i})+y_{i}\otimes(a\cdot x_{i})\big)\bullet(f_{j}\otimes
(e_{j}\diamond p)).
\end{align*}
If $r+\tau(r)$ is ass-invariant, i.e., for any $a\in A$,
$$
\sum_{i}x_{i}\otimes(a\cdot y_{i})+y_{i}\otimes(a\cdot x_{i})
=\sum_{i}(x_{i}\cdot a)\otimes y_{i}+(y_{i}\cdot a)\otimes x_{i},
$$
we get $\big((\fr_{\vdash}-\fr_{\dashv})(a\otimes p)\otimes\id
+\id\otimes\fl_{\dashv}(a\otimes p)\big)\big(\widetilde{r}-\tau(\widetilde{r})\big)=0$.
Similarly, we also have $\big(\fr_{\vdash}(a\otimes p)\otimes\id-\id\otimes(\fl_{\vdash}
-\fl_{\dashv})(a\otimes p)\big)\big(\widetilde{r}-\tau(\widetilde{r})\big)=0$.
Hence, $\widetilde{r}-\tau(\widetilde{r})$ is diass-invariant.
%

Finally, note that for any $1\leq s, t\leq n$,
$$
\varpi\Big(\sum_{j}e_{j}\otimes f_{j},\; e_{s}\otimes e_{t}\Big)
=\varpi(e_{t}, e_{s})=-\omega(e_{s}, e_{t})
=\varpi\Big(\sum_{j}f_{j}\otimes e_{j},\; e_{s}\otimes e_{t}\Big).
$$
That is, $\sum_{j}e_{j}\otimes f_{j}$ is skew-symmetric.
We get $\widetilde{r}$ is symmetric if $r$ is skew-symmetric. The proof is finished.
\end{proof}

We now give another main conclusion of this section.

\begin{thm}\label{thm:indu-sdibia}
Let $(A, \cdot, \Delta)$ be an ASI bialgebra, $(P, \diamond, \varpi)$ be a quadratic
perm algebra and $(A\otimes P, \dashv$, $\vdash, \Delta_{\dashv}, \Delta_{\vdash})$ be
the induced diassociative bialgebra from $(A, \cdot, \Delta)$ by $(P, \diamond, \varpi)$.
If $\Delta=\Delta_{r}$ is defined by Eq. \eqref{ass-cobo} for $r\in A\otimes A$, then
$(A\otimes P, \dashv, \vdash, \Delta_{\dashv}, \Delta_{\vdash})=(A\otimes P, \dashv,
\vdash, \Delta_{\dashv,\widetilde{r}}, \Delta_{\vdash,\widetilde{r}})$ as diassociative
bialgebras, where $\Delta_{\dashv,\widetilde{r}}, \Delta_{\vdash,\widetilde{r}}$ are
defined by Eq. \eqref{cobdi} and $\widetilde{r}$ is defined by Eq. \eqref{r-max}.
Therefore, we obtain
\begin{enumerate}\itemsep=0pt
\item[$(i)$]  $(A\otimes P, \dashv, \vdash, \Delta_{\dashv}, \Delta_{\vdash})$
     is coboundary if $(A, \cdot, \Delta)$ is coboundary;
\item[$(ii)$] $(A\otimes P, \dashv, \vdash, \Delta_{\dashv}, \Delta_{\vdash})$ is
     quasi-triangular if $(A, \cdot, \Delta)$ is quasi-triangular;
\item[$(iii)$] $(A\otimes P, \dashv, \vdash, \Delta_{\dashv}, \Delta_{\vdash})$ is
     triangular if $(A, \cdot, \Delta)$ is triangular;
\item[$(iv)$] $(A\otimes P, \dashv, \vdash, \Delta_{\dashv}, \Delta_{\vdash})$ is
     factorizable if $(A, \cdot, \Delta)$ is factorizable.
\end{enumerate}
\end{thm}

\begin{proof}
Let $r=\sum_{i}x_{i}\otimes y_{i}\in A\otimes A$. First, for any $a\in A$ and $p\in P$, we have
$$
\Delta_{\vdash}(a\otimes p)=\Big(\sum_{i}\big(x_{i}\otimes(a\cdot y_{i})
-(x_{i}\cdot a)\otimes y_{i}\big)\Big)\bullet\Big(\sum_{(p)}p_{(1)}\otimes p_{(2)}\Big),
$$
where $\nu_{\varpi}(p)=\sum_{(p)}p_{(1)}\otimes p_{(2)}$ and $\Delta_{r}(a)=
(\id\otimes\fl_{A}(a)-\fr_{A}(a)\otimes\id)(r)=\sum_{i}\big(x_{i}\otimes(a\cdot y_{i})
-(x_{i}\cdot a)\otimes y_{i}\big)$. On the other hand,
\begin{align*}
\Delta_{\vdash,\widetilde{r}}(a\otimes p)&=-\big((\fr_{\vdash}-\fr_{\dashv})(a\otimes p)
\otimes\id+\id\otimes\fl_{\dashv}(a\otimes p)\big)(\widetilde{r})\\
&=\sum_{i,j}\Big(\big((x_{i}\cdot a)\otimes y_{i}\big)\bullet\big((p\diamond e_{j})\otimes
f_{j}\big)-\big((x_{i}\cdot a)\otimes y_{i}\big)\bullet
\big((e_{j}\diamond p)\otimes f_{j}\big)\\[-5mm]
&\qquad\qquad-\big(x_{i}\otimes(a\cdot y_{i})\big)\bullet\big(e_{j}\otimes
(f_{j}\diamond p)\big)\Big),
\end{align*}
where $\{e_{1}, e_{2},\cdots, e_{n}\}$ is a basis of $P$ and $\{f_{1}, f_{2},\cdots,
f_{n}\}$ is the dual basis of $\{e_{1}, e_{2}, \cdots, e_{n}\}$ with respect to $\varpi(-,-)$.
For any basis elements $e_{s}, e_{t}\in P$, since
\begin{align*}
& \varpi\Big(\sum_{(p)}p_{(1)}\otimes p_{(2)},\ \
e_{s}\otimes e_{t}\Big)=\varpi(p,\; e_{s}\diamond e_{t}),\\[-2mm]
& \varpi\Big(\sum_{j}(p\diamond e_{j})\otimes f_{j},\ \ e_{s}\otimes e_{t}\Big)
=\varpi(p,\; e_{t}\diamond e_{s}-e_{s}\otimes e_{t}),\\[-2mm]
& \varpi\Big(\sum_{j}(e_{j}\diamond p)\otimes f_{j},\ \ e_{s}\otimes e_{t}\Big)
=\varpi(p,\; e_{t}\diamond e_{s}),\\[-2mm]
& \varpi\Big(\sum_{j}e_{j}\otimes(f_{j}\diamond p),\ \ e_{s}\otimes e_{t}\Big)
=-\varpi(p,\; e_{s}\otimes e_{t}),
\end{align*}
we get that $\sum_{(p)}p_{(1)}\otimes p_{(2)}=-\sum_{j}e_{j}\otimes(f_{j}\diamond p)
=\sum_{j}(e_{j}\diamond p)\otimes f_{j}-\sum_{j}(p\diamond e_{j})\otimes f_{j}$, and so that
$$
\Delta_{\vdash,\widetilde{r}}(a\otimes p)=\Big(\sum_{i}\big(x_{i}\otimes(a\cdot y_{i})
-(x_{i}\cdot a)\otimes y_{i}\big)\Big)\bullet\Big(\sum_{(p)}p_{(1)}\otimes p_{(2)}\Big)
=\Delta_{\vdash}(a\otimes p).
$$
Similarly, we also have $\Delta_{\dashv,\widetilde{r}}=\Delta_{\dashv}$.
Thus, $(A\otimes P, \dashv, \vdash, \Delta_{\dashv}, \Delta_{\vdash})=(A\otimes P, \dashv,
\vdash, \Delta_{\dashv,\widetilde{r}}, \Delta_{\vdash,\widetilde{r}})$ as diassociative
bialgebras. We obtain the conclusion $(i)$.

Second, $(ii)$ and $(iii)$ follow from Proposition \ref{pro:AYBE-DYBE}.
Finally, suppose $(A, \cdot, \Delta_{r})$ is factorizable. We get $\mathcal{I}=
r^{\sharp}+\tau(r)^{\sharp}: A^{\ast}\rightarrow A$ is an isomorphism of vector spaces.
We need to show that $\widetilde{\mathcal{I}}:=\widetilde{r}^{\sharp}
-\tau(\widetilde{r})^{\sharp}: (A\otimes P)^{\ast}\rightarrow A\otimes P$ is an
isomorphism of vector spaces. Denote $\kappa:=\sum_{j}e_{j}\otimes
f_{j}\in P\otimes P$. Define $\kappa^{\sharp}: P^{\ast}\rightarrow P$ by
$\langle\kappa^{\sharp}(\xi_{1}),\; \xi_{2}\rangle=\langle\kappa,\;
\xi_{1}\otimes\xi_{2}\rangle$, for any $\xi_{1}, \xi_{2}\in P^{\ast}$.
Then, one can check that $\kappa^{\sharp}$ is a linear isomorphism and
$\langle\kappa^{\sharp}(\xi_{1}),\; \xi_{2}\rangle=\langle\kappa,\;
\xi_{1}\otimes\xi_{2}\rangle=-\langle\kappa,\; \xi_{2}\otimes\xi_{1}\rangle
=\langle\kappa^{\sharp}(\xi_{2}),\; \xi_{1}\rangle$.
Therefore, for any $\xi_{1}, \xi_{2}\in P^{\ast}$ and $\eta_{1}, \eta_{2}\in A^{\ast}$,
\begin{align*}
\langle\widetilde{r}^{\sharp}(\eta_{1}\otimes\xi_{1}),\; \eta_{2}\otimes\xi_{2}\rangle
&=\sum_{i,j}\langle(\eta_{1}\otimes\xi_{1})\otimes(\eta_{2}\otimes\xi_{2}),\ \
(x_{i}\otimes e_{j})\otimes(y_{i}\otimes f_{j})\rangle\\[-2mm]
&=\Big(\sum_{j}\langle\eta_{1}, x_{i}\rangle\langle\eta_{2}, y_{i}\rangle\Big)
\Big(\sum_{i}\langle\xi_{1}, e_{j}\rangle\langle\xi_{2}, f_{j}\rangle\Big)\\[-2mm]
&=\langle r^{\sharp}(\eta_{1}),\; \eta_{2}\rangle
\langle\kappa^{\sharp}(\xi_{1}),\; \xi_{2}\rangle\\
&=\langle r^{\sharp}(\eta_{1})\otimes\kappa^{\sharp}(\xi_{1}),\;
\eta_{2}\otimes \xi_{2}\rangle.
\end{align*}
That is, $\widetilde{r}^{\sharp}=r^{\sharp}\otimes\kappa^{\sharp}$, Similarly,
$\tau(\widetilde{r})^{\sharp}=-\tau(r)^{\sharp}\otimes\kappa^{\sharp}$. Thus,
$\widetilde{\mathcal{I}}=\mathcal{I}\otimes\kappa^{\sharp}$ is an isomorphism of
vector spaces. The proof is completed.
\end{proof}

\begin{ex}\label{ex:ind-tribi}
Consider $2$-dimensional associative algebra $(A, \cdot)$, where vector space
$A={\rm span}_{\Bbbk}\{e_{1}$, $e_{2}\}$, the nonzero products are given by
$e_{1}\cdot e_{1}=e_{1}$ and $e_{1}\cdot e_{2}=e_{2}$. Denote $r=e_{2}\otimes e_{1}
-e_{1}\otimes e_{2}$. Then $r$ is a skew-symmetric solution of the $\AYBE$ in $(A, \cdot)$.
Thus, we get a triangular ASI bialgebra $(A, \cdot, \Delta)$, where the coporducts are
given by $\Delta(e_{1})=e_{2}\otimes e_{1}$ and $\Delta(e_{2})=e_{2}\otimes e_{2}$.
Let $(P, \diamond, \varpi)$ be the $2$-dimensional quadratic perm algebra given
in Example \ref{ex:dual-perm}. Then we get a $4$-dimensional diassociative algebra
$(A\otimes P, \dashv, \vdash)$, where the nonzero products are given by
\begin{align*}
& (e_{1}\otimes x_{2})\vdash(e_{1}\otimes x_{1})=e_{1}\otimes x_{1},&
& (e_{1}\otimes x_{2})\vdash(e_{2}\otimes x_{1})=e_{2}\otimes x_{1},\\
& (e_{1}\otimes x_{2})\vdash(e_{1}\otimes x_{2})=e_{1}\otimes x_{2},&
& (e_{1}\otimes x_{2})\vdash(e_{2}\otimes x_{2})=e_{2}\otimes x_{2},\\
& (e_{1}\otimes x_{1})\dashv(e_{1}\otimes x_{2})=e_{1}\otimes x_{1},&
& (e_{1}\otimes x_{1})\dashv(e_{2}\otimes x_{2})=e_{2}\otimes x_{1},\\
& (e_{1}\otimes x_{2})\dashv(e_{1}\otimes x_{2})=e_{1}\otimes x_{2},&
& (e_{1}\otimes x_{2})\dashv(e_{2}\otimes x_{2})=e_{2}\otimes x_{2}.
\end{align*}
Moreover, by Proposition \ref{pro:AYBE-DYBE}, we get a symmetric solution
$$
\widetilde{r}=(e_{1}\otimes x_{1})\otimes(e_{2}\otimes x_{2})
+(e_{2}\otimes x_{2})\otimes(e_{1}\otimes x_{1})
-(e_{1}\otimes x_{2})\otimes(e_{2}\otimes x_{1})
-(e_{2}\otimes x_{1})\otimes(e_{1}\otimes x_{2})
$$
of the $\DAYBE$ in $(A\otimes P, \dashv, \vdash)$. Thus, $\widetilde{r}$ induced
a triangular diassociative bialgebra structure on $(A\otimes P, \dashv, \vdash)$,
which is precisely the diassociative bialgebra given in Example \ref{ex:ind-diabi}.
\end{ex}

\begin{rmk}\label{rmk:nond}
Let $V$ be a vector space. Recall that an element $r\in V\otimes V$ is called {\bf
nondegenerate} if $r^{\sharp}: V^{\ast}\rightarrow V$ is a linear isomorphism.
Let $(A, \cdot)$ be an associative algebra and $(P, \diamond, \varpi)$ be a quadratic
perm algebra. By the proof of Theorem \ref{thm:indu-sdibia}, it is easy to see that
$\widetilde{r}\in(A\otimes P)\otimes(A\otimes P)$ is nondegenerate if
$r\in A\otimes A$ is nondegenerate.

Moreover, similar to the discussion in \cite{HL}, we can also construct a diassociative
bialgerba by the tensor product of perm bialgebra and a quadratic associative, and show that
this diassociative bialgerba is coboundary (resp. quasi-triangular, triangular, factorizable)
if the perm bialgebra is coboundary (resp. quasi-triangular, triangular, factorizable).
\end{rmk}

Let $(A, \cdot)$ be an associative algebra, $r\in A\otimes A$ and $(P, \diamond, \varpi)$ be a
quadratic perm algebra. By the proof of Theorem \ref{thm:indu-sdibia},
we have following commutative diagram:
$$
\xymatrix@C=2cm@R=0.6cm{
\txt{$r$ \\ {\tiny a solution of the $\AYBE$ in $(A, \cdot)$}\\
{\tiny such that $r+\tau(r)$ is ass-invariant}}
\ar[d]_{{\rm Pro.}~\ref{pro:AYBE-DYBE}}\ar[r]^{{\rm Pro.}~\ref{pro:quasass-bia}} &
\txt{$(A, \cdot, \Delta_{r})$ \\ {\tiny a quasi-triangular ASI bialgebra}}
\ar[d]^{{\rm Thm.}~\ref{thm:indu-sdibia}} \\
\txt{$\widetilde{r}$ \\ {\tiny a solution of the $\DAYBE$ in $(A\otimes P, \dashv, \vdash)$}\\
{\tiny such that $\widetilde{r}-\tau(\widetilde{r})$ is Leib-invariant}}
\ar[r]^{{\rm Pro.}~\ref{pro:quasi-di}} &
\txt{$(A\otimes P, \dashv, \vdash, \Delta_{\dashv,\widetilde{r}},
\Delta_{\vdash,\widetilde{r}})$ \\
{\tiny a quasi-triangular diassociative bialgebra}}}
$$
In particular, for a skew-symmetric solution $r$ of the $\AYBE$ in $(A, \cdot)$, we get
$\widetilde{r}$ is a symmetric solution of the $\DAYBE$ in $(A\otimes P, \dashv, \vdash)$,
$(A\otimes P, \dashv, \vdash, \Delta_{\dashv}, \Delta_{\vdash})=(A\otimes P, \dashv,
\vdash, \Delta_{\dashv,\widetilde{r}}, \Delta_{\vdash,\widetilde{r}})$
as triangular diassociative bialgebras, and the diagram is also commutative.
In \cite{Kup}, Kupershmidt found that the $\CYBE$ in tensor form on Lie algebras can
be converted into an $\mathcal{O}$-operator associated to the coregular representation.
This conclusion has been confirmed on various types of algebras.
Let $(A, \cdot)$ be an associative algebra and $(V, \kl, \kr)$ be a bimodule over
$(A, \cdot)$. Recall that a linear map $T: V\rightarrow A$ is called an {\bf
$\mathcal{O}$-operator of $(A, \cdot)$ associated to $(V, \kl, \kr)$} if for
any $v_{1}, v_{2}\in V$,
$$
T(v_{1})\cdot T(v_{2})=T\big(\kl(T(v_{1}))(v_{2})+\kr(T(v_{2}))(v_{1})\big).
$$

\begin{pro}[\cite{Bai}]\label{pro:o-ass}
Let $(A, \cdot)$ be an associative algebra and $r\in A\otimes A$ be skew-symmetric.
Then $r$ is a solution of the $\AYBE$ in $(A, \cdot)$ if and only if
$r^{\sharp}: A^{\ast}\rightarrow A$ is an $\mathcal{O}$-operator of $(A, \cdot)$
associated to the coregular module $(A^{\ast}, -\fr_{A}^{\ast}, -\fl_{A}^{\ast})$.
\end{pro}

The $\mathcal{O}$-operator of a diassociative algebra was considered in \cite{HLLZ,Lu}.
Recall that an {\bf $\mathcal{O}$-operator of a diassociative algebra $(D, \dashv, \vdash)$
associated to a bimodule $(V, \kl_{\dashv}, \kr_{\dashv}, \kl_{\vdash}, \kr_{\vdash})$}
is a linear map $T: V\rightarrow D$ such that
\begin{align*}
T(v_{1})\dashv T(v_{2})&=T\big(\kl_{\dashv}(T(v_{1}))(v_{2})
+\kr_{\dashv}(T(v_{2}))(v_{1})\big),\\
T(v_{1})\vdash T(v_{2})&=T\big(\kl_{\vdash}(T(v_{1}))(v_{2})
+\kr_{\vdash}(T(v_{2}))(v_{1})\big),
\end{align*}
for any $v_{1}, v_{2}\in V$.

\begin{pro}[\cite{HLLZ,Lu}]\label{pro:o-dia}
Let $(D, \dashv, \vdash)$ be a diassociative algebra and $r\in D\otimes D$ be symmetric.
Then $r$ is a solution of the $\DAYBE$ in $(D, \dashv, \vdash)$ if and only if
$r^{\sharp}$ is an $\mathcal{O}$-operator of $(D, \dashv, \vdash)$ associated to the
coregular bimodule $(D^{\ast}, \fr_{\vdash}^{\ast}-\fr_{\dashv}^{\ast},
-\fl_{\vdash}^{\ast}, -\fr_{\dashv}^{\ast}, \fl_{\dashv}^{\ast}-\fl_{\vdash}^{\ast})$.
\end{pro}

Let $(A, \cdot)$ be an associative algebra, $r\in A\otimes A$ and $(P, \diamond, \varpi)$ be a
quadratic perm algebra. Then we get a diassociative algebra $(A\otimes P, \dashv, \vdash)$.
By the proof of Theorem \ref{thm:indu-sdibia}, we also have
the following commutative diagram:
$$
\xymatrix@C=3cm@R=0.5cm{
\txt{$r$ \\ {\tiny a skew-symmetric solution} \\ {\tiny of the $\AYBE$ in $(A, \cdot)$}}
\ar[d]_-{{\rm Pro.}~\ref{pro:AYBE-DYBE}}\ar[r]^-{{\rm Pro.}~\ref{pro:o-ass}} &
\txt{$r^{\sharp}$\\ {\tiny an $\mathcal{O}$-operator of $(A, \cdot)$} \\
{\tiny associated to coregular bimodule}}
\ar[d]^-{\mbox{$-\otimes\kappa^{\sharp}$}} \\
\txt{$\widetilde{r}$ \\ {\tiny a symmetric solution} \\ {\tiny of the $\DAYBE$ in
$(A\otimes P, \dashv, \vdash)$}} \ar[r]^-{{\rm Pro.}~\ref{pro:o-dia}}
& \txt{$\widetilde{r}^{\sharp}=r^{\sharp}\otimes\kappa^{\sharp}$ \\
{\tiny an $\mathcal{O}$-operator of $(A\otimes P, \dashv, \vdash)$ } \\
{\tiny associated to coregular bimodule}}}
$$

At the end of this section, we give a simple example.

\begin{ex}\label{ex:ind-tri}
Let $(A, \cdot)$ be the $2$-dimensional associative algebra given in Example
\ref{ex:ind-tribi}, $r=e_{2}\otimes e_{1}-e_{1}\otimes e_{2}\in A\otimes A$ and
$(P, \diamond, \varpi)$ be the $2$-dimensional quadratic perm algebra given in Example
\ref{ex:dual-perm}. Then we get a $4$-dimensional diassociative algebra $(A\otimes P,
\dashv, \vdash)$ and a symmetric solution $\widetilde{r}$ of the $\DAYBE$ in
$(A\otimes P, \dashv, \vdash)$ in Example \ref{ex:ind-tribi}.
The skew-symmetric solution $r$ induces an $\mathcal{O}$-operator $r^{\sharp}: A^{\ast}
\rightarrow A$, which is given by
$$
r^{\sharp}(\xi_{1})=-e_{2},\qquad\qquad  r^{\sharp}(\xi_{2})=e_{1},
$$
where $\xi_{1}, \xi_{2}\in A^{\ast}$ is the dual basis of $e_{1}, e_{2}$. Denote
$\kappa=x_{2}\otimes x_{1}-x_{1}\otimes x_{2}$. Then the linear isomorphism
$\kappa^{\sharp}: P^{\ast}\rightarrow P$ is given by
$$
\kappa^{\sharp}(\eta_{1})=-x_{2},\qquad\qquad  \kappa^{\sharp}(\eta_{2})=x_{1},
$$
where $\eta_{1}, \eta_{2}\in P^{\ast}$ is the dual basis of $x_{1}, x_{2}$.
Moreover, by direct calculation, we get the $\mathcal{O}$-operator
$\widetilde{r}^{\sharp}: (A\otimes P)^{\ast}\rightarrow(A\otimes P)$ induced by
$\widetilde{r}$ is given by
$$
\widetilde{r}^{\sharp}(\xi_{1}\otimes\eta_{1})=e_{2}\otimes x_{2},\qquad
\widetilde{r}^{\sharp}(\xi_{1}\otimes\eta_{2})=-e_{2}\otimes x_{1},\qquad
\widetilde{r}^{\sharp}(\xi_{2}\otimes\eta_{1})=-e_{1}\otimes x_{2},\qquad
\widetilde{r}^{\sharp}(\xi_{2}\otimes\eta_{2})=e_{2}\otimes x_{2}.
$$
It is easy to see that $\widetilde{r}^{\sharp}(\xi_{i}\otimes\eta_{j})=
r^{\sharp}(\xi_{i})\otimes\kappa^{\sharp}(\eta_{j})$ for any $1\leq i, j\leq 2$.
\end{ex}

\section{From antisymmetric infinitesimal bialgebras to Leibniz bialgebras}\label{sec:LeibBi}
In this section, we show that the commutative diagram in Section \ref{sec:Pre}
is correct at the level of bialgebras. Moreover, by considering some special
solutions of the Yang-Baxter equations in associative algebras, Lie algebras,
diassociative algebras and Leibniz algebras, show that the commutative graph is also
correct at the level of triangular bialgebras. At the end of this section,
the close relationship between symplectic Lie algebras, symplectic Leibniz algebras,
symplectic associative algebras and symplectic diassociative algebras is considered.

\subsection{Yang-Baxter equation and triangular bialgebra structrues}\label{subsec:sym-tri}
Recall that  {\bf Lie coalgebra} $(\g, \delta)$ is a vector space $\g$ with a linear map
$\delta: \g\rightarrow\g\otimes\g$, such that $\tau\circ\delta=-\delta$ and $(\id\otimes\delta)
\circ\delta-(\tau\otimes\id)\circ(\id\otimes\delta)\circ\delta=(\delta\otimes\id)\circ\delta$.
Given an associative algebra, we can get a Lie algebra by the commutator.
Considering the dual case, for a coassociative coalgebra $(C, \Delta)$,
we also have a Lie coalgebra structure on $C$ by setting a coproduct
$\delta=\Delta-\tau\circ\Delta: C\rightarrow C\otimes C$. A {\bf Lie bialgebra} is a triple
$(\g, [-,-], \delta)$ such that $(\g, [-,-])$ is a Lie algebra, $(\g, \delta)$ is a Lie
coalgebra and the following compatibility condition holds:
$$
\delta([g_{1}, g_{2}])=(\ad_{\g}(g_{1})\otimes\id+\id\otimes\ad_{\g}(g_{1}))
(\delta(g_{2}))-(\ad_{\g}(g_{2})\otimes\id+\id\otimes\ad_{\g}(g_{2}))(\delta(g_{1})),
$$
where $\ad_{\g}(g_{1})(g_{2})=[g_{1}, g_{2}]$, for all $g_{1}, g_{2}\in\g$.

\begin{pro}[\cite{Bai}]\label{pro:ass-lie}
Let $(A, \cdot, \Delta)$ be an ASI bialgebra. If we define a bracket $[-,-]: A\otimes A
\rightarrow A$ by $[a_{1}, a_{2}]=a_{1}\cdot a_{2}-a_{2}\cdot a_{1}$ and
$\delta=\Delta-\tau\circ\Delta: A\rightarrow A\otimes A$, then $(A, [-,-], \delta)$ is
a Lie bialgebra, which is called {\bf the Lie bialgebra induced by $(A, \cdot, \Delta)$}.
\end{pro}

A Lie bialgebra $(\g, [-,-], \delta)$ is called {\bf coboundary}
if there exists an element $r\in\g\otimes\g$ such that $\delta=\delta_{r}$,
where
\begin{align}
\delta_{r}(g)=(\id\otimes\ad_{\g}(g)+\ad_{\g}(g)\otimes\id)(r), \label{lie-cobo}
\end{align}
for any $g\in\g$. Let $(\g, [-,-])$ be a Lie algebra. An element
$r=\sum_{i}x_{i}\otimes y_{i}\in\g\otimes\g$ is said to be {\bf Lie-invariant}
if $(\id\otimes\ad_{\g}(g)+\ad_{\g}(g)\otimes\id)(r)=0$ for all $g\in\g$. The equation
$$
\mathbf{C}_{r}:=[r_{12}, r_{13}]+[r_{13}, r_{23}]+[r_{12}, r_{23}]=0
$$
is called the {\bf classical Yang-Baxter equation} (or $\CYBE$) in $(\g, [-,-])$,
where $[r_{12}, r_{13}]=\sum_{i,j}[x_{i}, x_{j}]\otimes y_{i}\otimes y_{j}$,
$[r_{13}, r_{23}]=\sum_{i,j}x_{i}\otimes x_{j}\otimes[y_{i}, y_{j}]$ and
$[r_{12}, r_{23}]=\sum_{i,j}x_{i}\otimes[y_{i}, x_{j}]\otimes y_{j}$.

\begin{pro}[\cite{RS,LS}]\label{pro:lie-bia}
Let $(\g, [-,-])$ be a Lie algebra, $r\in\g\otimes\g$ and $\delta_{r}:
\g\rightarrow\g\otimes\g$ be the linear map defined by Eq. \eqref{lie-cobo}.
\begin{enumerate}\itemsep=0pt
\item[$(i)$] If $r$ is a skew-symmetric solution of the $\CYBE$ in $(\g, [-,-])$, then
     $(\g, [-,-], \delta_{r})$ is a Lie bialgebra, which is called a {\bf triangular Lie
     bialgebra} associated with $r$.
\item[$(ii)$] If $r$ is a solution of the $\CYBE$ in $(\g, [-,-])$ and $r+\tau(r)$ is
     $\ad_{\g}$-invariant, then $(\g, [-,-], \delta_{r})$ is a Lie bialgebra, which is
     called a {\bf quasi-triangular Lie bialgebra} associated with $r$.
\end{enumerate}
\end{pro}

A quasi-triangular Lie bialgebra $(\g, [-,-], \delta_{r})$ is called a {\bf factorizable
Lie bialgebra} if $\mathcal{I}=r^{\sharp}+\tau(r)^{\sharp}: \g^{\ast}\rightarrow \g$
is an isomorphism of vector spaces. The relationship between the solutions of the $\AYBE$ in
an associative algebra and the solutions of the $\CYBE$ in the induced Lie algebra has
been discussed in \cite{Hou1}.

\begin{pro}[\cite{Hou1}]\label{pro:ass-Lie-YBE}
Let $(A, \cdot)$ be an associative algebra and $(A, [-,-])$ be the induced Lie algebra
from $(A, \cdot)$. Suppose $r=\sum_{i}x_{i}\otimes y_{i}\in A\otimes A$ is a solution
of the $\AYBE$ in $(A, \cdot)$. If $r+\tau(r)$ is ass-invariant, then
$r$ is a solution of the $\CYBE$ in $(A, [-,-])$ and $r+\tau(r)$ is Lie-invariant.

In particular, each skew-symmetric solution of the $\AYBE$ in $(A, \cdot)$ is also a
skew-symmetric solution of the $\CYBE$ in $(A, [-,-])$.
\end{pro}

Thus, for the relationship between quasi-triangular (resp. triangular, factorizable)
ASI bialgebras and quasi-triangular (resp. triangular, factorizable) Lie bialgebras,
we have the following proposition.

\begin{pro}[\cite{Hou1}]\label{pro:qtAss-qtLie}
Let $(A, \cdot)$ be an associative algebra and $r\in A\otimes A$.
Suppose $(A, \cdot, \Delta_{r})$ is an ASI bialgebra and $(A, [-,-], \delta)$ is the
induced Lie bialgebra from $(A, \cdot, \Delta_{r})$, where $\Delta_{r}$ is given by
Eq. \eqref{ass-cobo}. If $r+\tau(r)$ is ass-invariant, then
$(A, [-,-], \delta)=(A, [-,-], \delta_{r})$ as Lie bialgebras, where
$\delta_{r}$ is defined by Eq. \eqref{lie-cobo}. Thus, we have
\begin{enumerate}\itemsep=0pt
\item[$(i)$]   $(A, [-,-], \delta)$ is quasi-triangular if $(A, \cdot, \Delta)$ is
               quasi-triangular;
\item[$(ii)$]  $(A, [-,-], \delta)$ is triangular if $(A, \cdot, \Delta)$ is triangular;
\item[$(iii)$] $(A, [-,-], \delta)$ is factorizable if $(A, \cdot, \Delta)$ is factorizable.
\end{enumerate}
\end{pro}

Here we will discuss in detail the connection between triangular diassociative
bialgebras and triangular Leibniz bialgebras.
Let $(L, \ast)$ be a Leibniz algebra, $V$ be a vector space and $\kll, \krr:
L\rightarrow\gl(V)$ be two linear maps. Then $(V, \kll, \krr)$ is called a {\bf
representation of $(L, \ast)$} if for any $x_{1}, x_{2}\in L$,
\begin{align*}
&\qquad\; \kll(x_{1}\ast x_{2})=\kll(x_{1})\circ\kll(x_{2})-\kll(x_{2})\circ\kll(x_{1}),\\
&\krr(x_{1})\circ\krr(x_{2})=\krr(x_{2}\ast x_{1})-\kll(x_{2})\circ\krr(x_{1})
=-\krr(x_{1})\circ\kll(x_{2}).
\end{align*}
Define the left multiplication map $\fll_{L}: L\rightarrow\gl(L)$ and the
right multiplication map $\frr_{L}: L\rightarrow\gl(L)$ by $\fll_{L}(x_{1})(x_{2})
=x_{1}\ast x_{2}$ and $\frr_{L}(x_{1})(x_{2})=x_{2}\ast x_{1}$ respectively
for all $x_{1}, x_{2}\in L$. Then $(L, \fll_{L}, \frr_{L})$ is a representation of
$(L, \ast)$, which is called the {\bf regular representation} of $(L, \ast)$.

We now consider the Leibniz bialgebras. Let $L$ be a vector space and
$\vartheta: L\rightarrow L\otimes L$ be a linear map. Then $(L, \vartheta)$ is called
a {\bf Leibniz coalgebra} if
\begin{align*}
(\id\otimes\vartheta)\circ\vartheta&=(\vartheta\otimes\id)\circ\vartheta
+(\tau\otimes\id)\circ(\id\otimes\vartheta)\circ\vartheta,\\
(\vartheta\otimes\id)\circ\vartheta&=(\id\otimes\vartheta)\circ\vartheta
+(\id\otimes\tau)\circ(\vartheta\otimes\id)\circ\vartheta.
\end{align*}

\begin{defi}\label{def:bialg}
A {\bf Leibniz bialgebra} is a triple $(L, \ast, \vartheta)$ containing
a Leibniz algebra $(L, \ast)$ and a Leibniz coalgebra $(L, \vartheta)$
such that for any $x_{1}, x_{2}\in L$,
\begin{align}
&\qquad\qquad\big(\fr_{L}(x_{1})\otimes\id\big)(\vartheta(x_{2}))
=\tau\Big(\big(\fr_{L}(x_{2})\otimes\id\big)\vartheta(x_{1})\Big),   \label{bialg2}\\
& \vartheta(x_{1}\ast x_{2})+\Big(\big(\id\otimes\fr_{A}(x_{2})-\fr_{A}(x_{2})\otimes\id
+\fl_{A}(x_{2})\otimes\id\big)(\id\otimes\id-\tau)\Big)(\vartheta(x_{1}))\label{bialg1}\\[-2mm]
&\qquad\qquad\qquad\qquad\qquad\qquad\qquad
+\big(\id\otimes\,\fl_{L}(x_{1})+\fl_{L}(x_{1})\otimes\id\big)\vartheta(x_{2})=0.  \nonumber
\end{align}
\end{defi}

Recently, it has been proven in papers \cite{HLLZ} and \cite{Lu} that every diassociative
bialgebra has a Leibniz bialgebra structure.

\begin{pro}[\cite{HLLZ,Lu}]\label{pro:diass-leib}
Let $(D, \dashv, \vdash, \Delta_{\dashv}, \Delta_{\vdash})$ be a diassociative algebra.
If we define a a new binary operation $\ast$ on $D$ by $d_{1}\ast d_{2}=
d_{1}\vdash d_{2}-d_{2}\dashv d_{1}$ and $\delta=\Delta_{\vdash}-\tau\circ\Delta_{\dashv}:
D\rightarrow D\otimes D$, then $(D, \ast, \delta)$ is a Leibniz bialgebra, which is
called {\bf the Leibniz bialgebra induced by $(D, \dashv, \vdash, \Delta_{\dashv},
\Delta_{\vdash})$}.
\end{pro}

Let $(L, \ast)$ be a Leibniz algebra. If there exists an element $r\in L\otimes L$
such that $(L, \ast, \vartheta_{r})$ is a Leibniz bialgebra, where $\vartheta_{r}:
L\rightarrow L\otimes L$ is given by
\begin{align}
\vartheta_{r}(x)=\big((\fll_{L}+\frr_{L})(x)\otimes\id
-\id\otimes\frr_{L}(x)\big)(r),            \label{cobLeb}
\end{align}
for any $x\in L$, then $(L, \ast, \vartheta_{r})$ is called a {\bf coboundary
Leibniz bialgebra}. An element $r\in L\otimes L$ is called {\bf Leib-invariant} if
$\vartheta_{r}(x)=0$ for all $x\in L$. The equation
$$
\mathbf{L}_{r}=r_{12}\ast r_{13}-r_{12}\ast r_{23}-r_{23}\ast r_{12}+r_{23}\ast r_{13}=0
$$
is called the {\bf Leibniz Yang-Baxter equation} (or $\LYBE$) in the Leibniz
algebra $(L, \ast)$, where $r_{12}\ast r_{13}=\sum_{i,j}(x_{i}\ast x_{j})\otimes y_{i}
\otimes y_{j}$, $r_{12}\ast r_{23}=\sum_{i,j}x_{i}\otimes(y_{i}\ast x_{j})\otimes y_{j}$,
$r_{23}\ast r_{12}=\sum_{i,j}x_{j}\otimes(x_{i}\ast y_{j})\otimes y_{i}$ and
$r_{23}\ast r_{13}=\sum_{i,j}x_{j}\otimes x_{i}\otimes(y_{i}\ast y_{j})$.

\begin{pro}[\cite{BLST}]\label{pro:sLib-bia}
Let $(L, \ast)$ be a Leibniz algebra, $r\in L\otimes L$ and $\vartheta_{r}:
L\rightarrow L\otimes L$ be the linear map defined by Eq. \eqref{cobLeb}.
\begin{enumerate}
\item[$(i)$] If $r$ is a solution of the $\LYBE$ in $(L, \ast)$ and $r-\tau(r)$ is
     Leib-invariant, then $(L, \ast, \vartheta_{r})$ is a Leibniz bialgebra,
     which is called a {\bf quasi-triangular Leibniz bialgebra} associated with $r$.
\item[$(ii)$] If $r$ is a symmetric solution of the $\LYBE$ in $(L, \ast)$, then $(L, \ast,
     \vartheta_{r})$ is a Leibniz bialgebra, which is called a {\bf triangular Leibniz
     bialgebra} associated with $r$.
\end{enumerate}
\end{pro}

A quasi-triangular Leibniz bialgebra $(L, \ast, \vartheta_{r})$ is called a {\bf factorizable
Leibniz bialgebra} if $\mathcal{I}=r^{\sharp}-\tau(r)^{\sharp}: L^{\ast}\rightarrow L$
is an isomorphism of vector spaces.
Following, we will examine the relationship between the solutions of $\DAYBE$ and the solutions
of $\LYBE$ in the induced Leibniz algebra to provide the connection between these special
diassociative bialgebras and the induced Leibniz bialgebras.

\begin{pro}\label{pro:diass-Leib-YBE}
Let $(D, \dashv, \vdash)$ be a diassociative algebra and $(D, \ast)$ be the induced
Leibniz algebra from $(D, \dashv, \vdash)$. If $r=\sum_{i}x_{i}\otimes y_{i}\in
D\otimes D$ is a symmetric solution of the $\DAYBE$ in $(D, \dashv, \vdash)$, then
$r$ is also a symmetric solution of the $\LYBE$ in $(D, \ast)$.
\end{pro}

\begin{proof}
Note that $\mathbf{L}_{r}=r_{12}\ast r_{13}-r_{12}\ast r_{23}-r_{23}\ast r_{12}
+r_{23}\ast r_{13}=r_{12}\vdash r_{13}-r_{13}\dashv r_{12}-r_{12}\vdash r_{23}
+r_{23}\dashv r_{12}-r_{23}\vdash r_{12}+r_{12}\dashv r_{23}+r_{23}\vdash r_{13}
-r_{13}\dashv r_{23}$. If $r$ is symmetric, then we get
$$
\mathbf{L}_{r}=(\tau\otimes\id)(\mathbf{D}_{r})
+(\tau\otimes\id)((\id\otimes\tau)(\mathbf{D}_{r})).
$$
That is, $r$ is a symmetric solution of the $\LYBE$ in $(D, \ast)$ if $r$ is a
symmetric solution of the $\DAYBE$ in $(D, \dashv, \vdash)$.
\end{proof}

Thus, considering the triangular diassociative bialgebras and triangular Leibniz bialgebras,
we have

\begin{pro}\label{pro:qtdiAss-qtLeib}
Let $(D, \dashv, \vdash)$ be a diassociative algebra and $r\in D\otimes D$.
Suppose $(D, \dashv, \vdash, \Delta_{\dashv,r}, \Delta_{\dashv,r})$ is a diassociative
bialgebra and $(D, \ast, \vartheta)$ is the induced Leibniz bialgebra, where
$\Delta_{\dashv,r}, \Delta_{\dashv,r}$ are given by
Eq. \eqref{cobdi}. If $r$ is symmetric, then $(D, \ast, \vartheta)=(D, \ast,
\vartheta_{r})$ as Leibniz bialgebras, where $\vartheta_{r}$ is defined by Eq.
\eqref{cobLeb}. Thus, we obtain that the induced Leibniz bialgebra $(D, \ast, \vartheta)$
is triangular if $(D, \dashv, \vdash, \Delta_{\dashv,r}, \Delta_{\dashv,r})$ is
a triangular diassociative algebra.
\end{pro}

\begin{proof}
Let $r=\sum_{i}x_{i}\otimes y_{i}$. If $r$ is symmetric,
then for any $d\in D$, we have
\begin{align*}
\vartheta(d)&=\Delta_{\vdash,r}(a)-\tau(\Delta_{\dashv,r}(a))\\
&=\sum_{i}\Big((x_{i}\vdash d)\otimes y_{i}-(x_{i}\dashv d)\otimes y_{i}
+x_{i}\otimes(d\dashv y_{i})\\[-5mm]
&\qquad\qquad -y_{i}\otimes(x_{i}\vdash d)+(d\vdash y_{i})\otimes x_{i}
-(d\dashv y_{i})\otimes x_{i}\Big)\\
&=\big((\fll_{D}+\frr_{D})(d)\otimes\id)-\id\otimes\frr_{D}(d)\big)(r)\\
&=\vartheta_{r}(d).
\end{align*}
Thus, $(D, \ast, \vartheta)=(D, \ast, \vartheta_{r})$ as Leibniz bialgebras.
Finally, by Proposition \ref{pro:diass-Leib-YBE}, we get that $(D, \ast, \vartheta)$
is a triangular Leibniz bialgebra if $(D, \dashv, \vdash, \Delta_{\dashv,r},
\Delta_{\dashv,r})$ is triangular.
\end{proof}

Recently, in \cite{HL}, we have considered the construction problem of Leibniz bialgebras
from Lie bialgebras and perm bialgebras. More precisely, We have shown that the tensor
product of a Lie bialgebra and a quadratic perm algebra has a Leibniz bialgebra structure,
and this Leibniz bialgebra structure is coboundary (resp. quasi-triangular, triangular,
factorizable) if the original Lie bialgebra is coboundary (resp. quasi-triangular,
triangular, factorizable).

\begin{thm}[\cite{HL}]\label{thm:Lie+perm=L}
Let $(\g, [-,-], \delta)$ be a Lie bialgebra, $(P, \diamond, \varpi)$ be a quadratic
perm algebra and $(\g\otimes P, \ast)$ be the induced Leibniz algebra. Define
a linear map $\vartheta: \g\otimes P\rightarrow(\g\otimes P)\otimes(\g\otimes P)$ by
$$
\vartheta(g\otimes p)=\delta(g)\bullet\nu_{\varpi}(p)
:=\sum_{(g)}\sum_{(p)}(g_{(1)}\otimes p_{(1)})
\otimes(g_{(2)}\otimes p_{(2)}),
$$
for any $g\in\g$ and $p\in P$, where $\delta(g)=\sum_{(g)}g_{(1)}\otimes g_{(2)}$
and $\nu_{\varpi}(p)=\sum_{(p)}p_{(1)}\otimes p_{(2)}$ in the Sweedler notation.
Then $(\g\otimes P, \ast, \vartheta)$ is a Leibniz bialgebra, which is called the
{\bf Leibniz bialgebra induced from $(\g, [-,-], \delta)$ by $(P, \diamond, \varpi)$}.
\end{thm}

Let $(A, \cdot, \Delta)$ be an ASI bialgebra. By Theorem \ref{thm:permbia-dia}, we get a
diassociative bialgebra $(A\otimes P, \dashv, \vdash, \Delta_{\dashv},\Delta_{\vdash})$,
where $\Delta_{\dashv}(a\otimes p)=\Delta(a)\bullet\tau(\nu_{\varpi}(p))$,
$\Delta_{\vdash}(a\otimes p)=\Delta(a)\bullet\nu_{\varpi}(p)$,
$(a_{1}\otimes p_{1})\dashv(a_{2}\otimes p_{2})=(a_{1}\cdot a_{2})\otimes
(p_{2}\diamond p_{1})$ and $(a_{1}\otimes p_{1})\vdash(a_{2}\otimes p_{2})
=(a_{1}\cdot a_{2})\otimes(p_{1}\diamond p_{2})$ for any $a, a_{1}, a_{2}\in A$,
$p, p_{1}, p_{2}\in P$. By Proposition \ref{pro:diass-leib}, this diassociative bialgebra
induced a Leibniz bialgebra $(A\otimes P, \ast, \vartheta)$, where
\begin{align*}
\vartheta(a\otimes p)&
=(\Delta(a)-\tau(\Delta(a)))\bullet\nu_{\varpi}(p),\\
(a_{1}\otimes p_{1})\ast(a_{2}\otimes p_{2})&=(a_{1}\cdot a_{2}-a_{2}\cdot a_{1})
\otimes(p_{1}\diamond p_{2}),
\end{align*}
On the other hand, by Proposition \ref{pro:ass-lie}, the ASI bialgebra $(A, \cdot, \Delta)$
induces a Lie bialgebra $(A, [-,-], \delta)$, where $\delta=\Delta-\tau\circ\Delta$
and $[a_{1}, a_{2}]= a_{1}\cdot a_{2}-a_{2}\cdot a_{1}$ for any $a_{1}, a_{2}\in A$.
Moreover, by Theorem \ref{thm:Lie+perm=L}, we get a Leibniz bialgebra
$(A\otimes P, \ast', \vartheta')$. It is easy to see that $\ast=\ast'$ and $\vartheta
=\vartheta'$. Thus, we have the following commutative diagram:
$$
\xymatrix@C=2cm@R=0.7cm{
\txt{$(A, \cdot, \Delta)$ \\ {\tiny an ASI bialgebra}}
\ar[d]_{{\rm Pro.}~\ref{pro:ass-lie}} \ar[r]^{{\rm Thm.}~\ref{thm:permbia-dia}}
&\txt{$(A\otimes P, \dashv, \vdash, \Delta_{\dashv},\Delta_{\vdash})$\\
{\tiny a diassociative bialgebra}}\ar[d]^{{\rm Pro.}~\ref{pro:diass-leib}} \\
\txt{$(A, [-,-], \delta)$ \\ {\tiny a Lie bialgebra}}
\ar[r]^{{\rm Thm.}~\ref{thm:Lie+perm=L}}
& \txt{$(A\otimes P, \ast, \vartheta)$ \\ {\tiny a Leibniz bialgebra}}}
$$

For the solutions of the $\CYBE$ in a Lie algebra and the solutions of the $\LYBE$ in
the induced Leibniz algebra, we have:

\begin{pro}[\cite{HL}]\label{pro:CYBE-LYBE}
Let $(\g, [-,-])$ be a Lie algebra, $(P, \diamond, \varpi)$ be a quadratic
perm algebra and $(\g\otimes P, \ast)$ be the induced Leibniz algebra.
Suppose that $r=\sum_{i}x_{i}\otimes y_{i}\in\g\otimes\g$ is a solution of the $\CYBE$
in $(\g, [-,-])$, $r+\tau(r)$ is Lie-invariant, $\{e_{1}, e_{2},\cdots, e_{n}\}$ is
a basis of $P$ and $\{f_{1}, f_{2},\cdots, f_{n}\}$ is the dual basis of $\{e_{1},
e_{2},\cdots, e_{n}\}$ with respect to $\varpi(-,-)$. Then
\begin{align}
\widetilde{r}=\sum_{i, j}(x_{i}\otimes e_{j})\otimes(y_{i}\otimes f_{j})
\in(\g\otimes P)\otimes(\g\otimes P)  \label{rr-max}
\end{align}
is a solution of the $\LYBE$ in $(\g\otimes P, \ast)$, and
$\widetilde{r}-\tau(\widetilde{r})$ is Leib-invariant.

In particular, $\widetilde{r}$ is a symmetric solution of the $\LYBE$ in
$(\g\otimes P, \ast)$ if $r$ is a skew-symmetric solution of the $\CYBE$ in $(\g, [-,-])$.
\end{pro}

Thus, we can obtain the relationship between the coboundary (resp. quasi-triangular,
triangular, factorizable) Lie bialgebras and the coboundary (resp. quasi-triangular,
triangular, factorizable) Leibniz bialgebras.

\begin{thm}[\cite{HL}]\label{thm:indu-triLib}
Let $(\g, [-,-], \delta)$ be a Lie bialgebra, $(P, \diamond, \varpi)$ be a quadratic perm
algebra and $(\g\otimes P, \ast, \vartheta)$ be the induced Leibniz bialgebra from
$(\g, [-,-], \delta)$ by $(P, \diamond, \varpi)$. If $\delta=\delta_{r}$ for some
$r\in\g\otimes\g$, then $\vartheta=\vartheta_{\widetilde{r}}$, where $\widetilde{r}$ is
given by Eq. \eqref{rr-max}. Therefore, we obtain
\begin{enumerate}\itemsep=0pt
\item[$(i)$]  $(\g\otimes P, \ast, \vartheta)$ is coboundary if
     $(\g, [-,-], \delta)$ is coboundary;
\item[$(ii)$] $(\g\otimes P, \ast, \vartheta)$ is quasi-triangular if
     $(\g, [-,-], \delta)$ is quasi-triangular;
\item[$(iii)$] $(\g\otimes P, \ast, \vartheta)$ is triangular if
     $(\g, [-,-], \delta)$ is triangular;
\item[$(iv)$] $(\g\otimes P, \ast, \vartheta)$ is factorizable if
     $(\g, [-,-], \delta)$ is factorizable.
\end{enumerate}
\end{thm}

Based on the above results, it is natural for us to ask whether the following commutative
diagram holds true?
{\small
\begin{displaymath}
\xymatrix@R=0.4cm@C=-0.2cm{
&\txt{{\small $(A, \cdot, \Delta_{r})$ }\\ {\tiny a triangular ASI bialgebra}}
\ar@{->}[rr]^{{\rm Thm.}~\ref{thm:indu-sdibia}}
\ar@{->}[ld]^{{\rm Pro.}~\ref{pro:qtAss-qtLie}}
\ar@{<.}[dd]_(0.7){{\rm Pro.}~\ref{pro:quasass-bia}}&
&\txt{{\small $(A\otimes P, \dashv, \vdash, \Delta_{\dashv,\widetilde{r}},
\Delta_{\vdash,\widetilde{r}})$} \\ {\tiny a triangular diassociative bialgebra}}
\ar@{->}[ld]^{{\rm Pro.}~\ref{pro:qtdiAss-qtLeib}}
\ar@{<-}[dd]^{{\rm Pro.}~\ref{pro:quasi-di}}&\\
\txt{{\small $(A, [-,-], \vartheta_{r})$} \\ {\tiny a triangular Lie bialgebra}}
\ar@{<-}[dd]^{{\rm Pro.}~\ref{pro:lie-bia}}
\ar@{->}[rr]^(0.6){{\rm Thm.}~\ref{thm:indu-triLib}}&
&\txt{{\small $(A\otimes P, \ast, \vartheta_{\widetilde{r}})$}\\
{\tiny a triangular Leibniz bialgebra}}
\ar@{<-}[dd]_(0.7){{\rm Pro.}~\ref{pro:sLib-bia}} \\
&\txt{{\small $r$}\\ {\tiny a skew-symmetric solution} \\
{\tiny of the $\AYBE$ in $(A, \cdot)$}}
\ar@{.>}[rr]^(0.6){{\rm Pro.}~\ref{pro:AYBE-DYBE}}&
&\txt{{\small $\widetilde{r}$} \\  {\tiny a symmetric solution} \\
{\tiny of the $\DAYBE$ in $(A\otimes P, \dashv, \vdash)$}}\\
\txt{{\small $r$}\\ {\tiny a skew-symmetric solution} \\
{\tiny of the $\CYBE$ in $(A, [-,-])$}}
\ar@{<.}[ru]_{{\rm Pro.}~\ref{pro:ass-Lie-YBE}}
\ar@{->}[rr]^{{\rm Pro.}~\ref{pro:CYBE-LYBE}}&
&\txt{{\small $\widetilde{r}$}\\ {\tiny a symmetric solution} \\
{\tiny of the $\LYBE$ in $(A\otimes P, \ast)$}}
\ar@{<-}[ru]_{{\rm Pro.}~\ref{pro:diass-Leib-YBE}}}
\end{displaymath}
}\\
Indeed, It can be inferred from Theorem \ref{thm:indu-triLib}, Theorem \ref{thm:indu-sdibia},
Proposition \ref{pro:qtAss-qtLie} and Proposition \ref{pro:qtdiAss-qtLeib} that the front,
back, left, and right sides of the diagram are commutative respectively.
The bottom is naturally commutative. Direct verification shows that the surface above is also
commutative. So the commutative diagram hold.

\begin{rmk}\label{rmk:commdiagram}
In \cite{Hou1}, we have constructed a Lie bialgebra by a dendriform $\mathrm{D}$-bialgebra,
and discussed in detail the relationship between pre-Lie bialgebras, antisymmetric
infinitesimal bialgebras, dendriform $\mathrm{D}$-bialgebras and Lie bialgebras.
Let $(D, \prec, \succ, \theta_{\prec}, \theta_{\succ})$ be a dendriform
$\mathrm{D}$-bialgebra, $(P, \diamond, \varpi)$ and $(P', \diamond', \varpi')$ be two
quadratic perm algebras. Based on the conclusion here and the conclusion in \cite{Hou1},
we have the following commutative diagram:
$$
\xymatrix@C=1.5cm@R=0.7cm{
\txt{$(D, \prec, \succ, \theta_{\prec}, \theta_{\succ})$ \\
{\tiny a dendriform $\mathrm{D}$-bialgebra}}\ar[d]\ar[r]
& \txt{$(D\otimes P, \cdot, \Delta)$ \\ {\tiny an ASI bialgebra}}\ar[d]\ar[r]
& \txt{$((D\otimes P)\otimes P', \dashv, \vdash, \Delta_{\dashv},\Delta_{\vdash})$\\
{\tiny a diassociative bialgebra}}\ar[d] \\
\txt{$(D, \circ, \theta)$ \\ {\tiny a pre-Lie bialgebra}}\ar[r]
& \txt{$(D\otimes P, [-,-], \delta)$ \\ {\tiny a Lie bialgebra}}\ar[r]
& \txt{$((D\otimes P)\otimes P', \ast, \vartheta)$ \\ {\tiny a Leibniz bialgebra}}}
$$
Moreover, if the dendriform $\mathrm{D}$-bialgebra $(D, \prec, \succ, \theta_{\prec},
\theta_{\succ})$ is triangular, then all the bialgebra structures in the commutative
diagram above are triangular.
\end{rmk}

\begin{ex}\label{ex:ind-tridiam}
Let $(A, \cdot)$ be the $2$-dimensional associative algebra given in Example
\ref{ex:ind-tribi}, $r=e_{2}\otimes e_{1}-e_{1}\otimes e_{2}\in A\otimes A$ and
$(P, \diamond, \varpi)$ be the $2$-dimensional quadratic perm algebra given in Example
\ref{ex:dual-perm}. Since $r$ is a skew-symmetric solution of the $\AYBE$ in $(A, \cdot)$,
we have a triangular ASI bialgebra $(A, \cdot, \Delta_{r})$, where $\Delta(e_{1})
=e_{2}\otimes e_{1}$ and $\Delta(e_{2})=e_{2}\otimes e_{2}$.
Moreover, we get a $4$-dimensional diassociative algebra $(A\otimes P,
\dashv, \vdash)$ and a symmetric solution $\widetilde{r}$ of the $\DAYBE$ in
$(A\otimes P, \dashv, \vdash)$ in Example \ref{ex:ind-tribi}. Therefore, we
have a triangular diassociative algebra $(A\otimes P, \dashv, \vdash,
\Delta_{\dashv,\widetilde{r}}, \Delta_{\vdash,\widetilde{r}})$, which is given
in Example \ref{ex:ind-diabi}.

On the other hand, by the associative algebra $(A, \cdot)$, we can obtain a Lie algebra
$(A, [-,-])$, where $[e_{1}, e_{2}]=e_{2}=-[e_{2}, e_{1}]$, and one can check that
$r$ is also a skew-symmetric solution of the $\CYBE$ in $(A, [-,-])$. Thus, $r$ induces a
triangular Lie bialgebra $(A, [-,-], \delta_{r})$, where $\delta_{r}(e_{1})=
e_{2}\otimes e_{1}-e_{1}\otimes e_{2}$ and $\delta_{r}(e_{2})=0$. By this Lie bialgebra
and quadratic perm algebra $(P, \diamond, \varpi)$, we get a Leibniz bialgebra
$(A\otimes P, \ast, \vartheta)$, where
\begin{align*}
& (e_{1}\otimes x_{2})\ast(e_{2}\otimes x_{1})=e_{2}\otimes x_{1}
=-(e_{2}\otimes x_{2})\ast(e_{1}\otimes x_{1}),\\
& (e_{1}\otimes x_{2})\ast(e_{2}\otimes x_{2})=e_{2}\otimes x_{2}
=-(e_{2}\otimes x_{2})\ast(e_{1}\otimes x_{2}),\\
& \vartheta(e_{1}\otimes x_{1})=(e_{2}\otimes x_{1})\otimes(e_{1}\otimes x_{1})
-(e_{1}\otimes x_{1})\otimes(e_{2}\otimes x_{1}),\\
& \vartheta(e_{1}\otimes x_{2})=(e_{2}\otimes x_{1})\otimes(e_{1}\otimes x_{2})
-(e_{1}\otimes x_{1})\otimes(e_{2}\otimes x_{2}),
\end{align*}
and $\vartheta(e_{2}\otimes x_{1})=\vartheta(e_{2}\otimes x_{2})=0$.
One can check that $\widetilde{r}$ is a symmetric solution of the $\LYBE$ in
$(A\otimes P, \ast)$, the Leibniz bialgerba $(A\otimes P, \ast, \vartheta)$
is exactly the triangular Leibniz bialgerba associated with $\widetilde{r}$,
and it is also the Leibniz bialgerba induced by the diassociative algebra
$(A\otimes P, \dashv, \vdash, \Delta_{\dashv,\widetilde{r}}, \Delta_{\vdash,\widetilde{r}})$.
\end{ex}

\subsection{Symmetric solutions of Yang-Baxter equation and $\mathcal{O}$-operators}
\label{subsec:sym-o}
It is known that there is a close relationship between $\mathcal{O}$-operators and the
classical Yang-Baxter equation on Lie algebras, as the former is an operator form of a
solution of the $\CYBE$ (or classical $r$-matrix). Let $(\g, [-,-])$ be a Lie algebra
and $(V, \rho)$ be a representation of $(\g, [-,-])$.
Recall that a linear map $T: V\rightarrow \g$ is called an {\bf $\mathcal{O}$-operator of
$(\g, [-,-])$ associated to $(V, \rho)$} if for any $v_{1}, v_{2}\in V$,
$$
[P(v_{1}), P(v_{2})]=P\big(\rho(P(v_{1}))(v_{2})-\rho(P(v_{2}))(v_{1})\big).
$$

\begin{pro}[\cite{CP,Kup}]\label{pro:o-lie}
Let $(\g, [-,-])$ be a Lie algebra and $r\in\g\otimes\g$ be skew-symmetric.
Then $r$ is a solution of the $\CYBE$ in $(\g, [-,-])$ if and only if $r^{\sharp}$
is an $\mathcal{O}$-operator of $(\g, [-,-])$ associated to the coregular
representation $(\g^{\ast}, \ad_{\g}^{\ast})$.
\end{pro}

Let $(A, \cdot)$ be an associative algebra and $(V, \kl, \kr)$ be a bimodule over it.
Consider the induced Lie algebra $(A, [-,-])$ by the commutator. One can check that
$(V, \kl-\kr)$ is a representation of $(A, [-,-])$.

\begin{pro}\label{pro:o-ass-lie}
Let $(A, \cdot)$ be an associative algebra and $(A, [-,-])$ be the induced Lie algebra.
If $T: V\rightarrow A$ is an $\mathcal{O}$-operator of $(A, \ast)$ associated to
bimodule $(V, \kl, \kr)$, then $T$ is also an $\mathcal{O}$-operator of $(A, [-,-])$
associated to representation $(V, \kl-\kr)$.
\end{pro}

\begin{proof}
Let $T: V\rightarrow A$ be an $\mathcal{O}$-operator of $(A, \ast)$ associated
to $(V, \kl, \kr)$. That is, $T(v_{1})\cdot T(v_{2})=T\big(\kl(T(v_{1}))(v_{2})
+\kr(T(v_{2}))(v_{1})\big)$ for any $v_{1}, v_{2}\in V$. Thus, we have
\begin{align*}
&\; [T(v_{1}),\; T(v_{2})]-T\big((\kl-\kr)(T(v_{1}))(v_{2})
-(\kl-\kr)(T(v_{2}))(v_{1})\big)\\
=&\; T(v_{1})\cdot T(v_{2})-T(v_{2})\cdot T(v_{1})
-T\big(\kl(T(\xi_{1}))(\xi_{2})-\kr(T(v_{1}))(v_{2})
-\kl(T(v_{2}))(v_{1})+\kr(T(v_{2}))(v_{1})\big)\\
=&\; 0.
\end{align*}
This means that $T$ is an $\mathcal{O}$-operator of $(A, [-,-])$ associated to $(V, \kl-\kr)$.
\end{proof}

In particular, if $(V, \kl, \kr)$ is the coregular bimodule $(A^{\ast}, -\fr_{A}^{\ast},
-\fl_{A}^{\ast})$, then $(A^{\ast}, \fl_{A}^{\ast}-\fr_{A}^{\ast})$ is the coregular
representation of $(A, [-,-])$, and $T: A^{\ast}\rightarrow A$ is the $\mathcal{O}$-operator
of $(A, [-,-])$ associated to the coregular representation.
Let $(L, \ast)$ be a Leibniz algebra and $(V, \bar{\kl}, \bar{\kr})$ be a representation of
$(L, \ast)$. Recall that a linear map $T: V\rightarrow L$ is called an {\bf
$\mathcal{O}$-operator of $(L, \ast)$ associated to $(V, \bar{\kl}, \bar{\kr})$} if for
any $v_{1}, v_{2}\in V$,
$$
T(v_{1})\ast T(v_{2})=T\big(\kl(T(v_{1}))(v_{2})+\kr(T(v_{2}))(v_{1})\big).
$$

\begin{pro}[\cite{TS,BLST}]\label{pro:o-leib}
Let $(L, \ast)$ be a Leibniz algebra and $r\in L\otimes L$ be symmetric.
Then $r$ is a solution of the $\LYBE$ in $(L, \ast)$ if and only if
$r^{\sharp}: L^{\ast}\rightarrow L$ is an $\mathcal{O}$-operator of $(L, \ast)$
associated to the coregular representation $(L^{\ast}, \bar{\fl}_{L}^{\ast},
-\bar{\fl}_{L}^{\ast}-\bar{\fr}_{L}^{\ast})$.
\end{pro}

Let $(\g, [-,-])$ be a Lie algebra and $(P, \diamond, \varpi)$ be a quadratic
perm algebra. Then we get a Leibniz algebra structure on the tensor product $\g\otimes P$.
In \cite{HL}, we have discussed the relationship between $\mathcal{O}$-operators on
Lie algebras and $\mathcal{O}$-operators on the induced Leibniz algebras, and presented the
following commutative diagram:
$$
\xymatrix@C=3cm@R=0.5cm{
\txt{$r$ \\ {\tiny a skew-symmetric solution} \\ {\tiny of the $\CYBE$ in $(\g, [-,-])$}}
\ar[d]_-{{\rm Pro.}~\ref{pro:CYBE-LYBE}}\ar[r]^-{{\rm Pro.}~\ref{pro:o-lie}} &
\txt{$r^{\sharp}$\\ {\tiny an $\mathcal{O}$-operator of $(\g, [-,-])$} \\
{\tiny associated to coregular representation}}
\ar[d]^-{\mbox{$-\otimes\kappa^{\sharp}$}} \\
\txt{$\widetilde{r}$ \\ {\tiny a symmetric solution} \\ {\tiny of the $\LYBE$ in
$(\g\otimes P, \ast)$}} \ar[r]^-{{\rm Pro.}~\ref{pro:o-leib}}
& \txt{$\widetilde{r}^{\sharp}=r^{\sharp}\otimes\kappa^{\sharp}$ \\
{\tiny an $\mathcal{O}$-operator of $(\g\otimes P, \ast)$ } \\
{\tiny associated to coregular representation}}}
$$

Let $(D, \dashv, \vdash)$ be a diassociative algebra and $(V, \kl_{\dashv}, \kr_{\dashv},
\kl_{\vdash}, \kr_{\vdash})$ be a bimodule over it. Then $(V, \kl_{\vdash}-\kr_{\dashv},
\kr_{\vdash}-\kl_{\dashv})$ is a representation of the induced Leibniz algebra $(D, \ast)$.
Moreover, similar to the proof of Proposition \ref{pro:o-ass-lie} we get the following
proposition by direct verification.

\begin{pro}\label{pro:o-diass-leib}
Let $(D, \dashv, \vdash)$ be a diassociative algebra and $(D, \ast)$ be the induced
Leibniz algebra. If $T: V\rightarrow D$ is an $\mathcal{O}$-operator of $(D, \dashv,
\vdash)$ associated to bimodule $(V, \kl_{\dashv}, \kr_{\dashv}, \kl_{\vdash}, \kr_{\vdash})$,
then $T$ is also an $\mathcal{O}$-operator of $(D, \ast)$ associated to
representation $(V, \kl_{\vdash}-\kr_{\dashv}, \kr_{\vdash}-\kl_{\dashv})$.
\end{pro}

In particular, if $(V, \kl_{\dashv}, \kr_{\dashv}, \kl_{\vdash}, \kr_{\vdash})$ is the
coregular representation $(D^{\ast}, \fr_{\vdash}^{\ast} -\fr_{\dashv}^{\ast},
-\fl_{\vdash}^{\ast}, -\fr_{\dashv}^{\ast}, \fl_{\dashv}^{\ast}-\fl_{\vdash}^{\ast})$,
then $(D^{\ast}, \fl_{\vdash}^{\ast}-\fr_{\dashv}^{\ast}, \fl_{\dashv}^{\ast}
+\fr_{\dashv}^{\ast}-\fl_{\vdash}^{\ast}-\fr_{\vdash}^{\ast})$ is exactly
the coregular representation of $(D, \ast)$, and $T: D^{\ast}\rightarrow D$ is the
$\mathcal{O}$-operator of $(D, \ast)$ associated to the coregular representation.
Thus, we get the following commutative diagram:
{\small
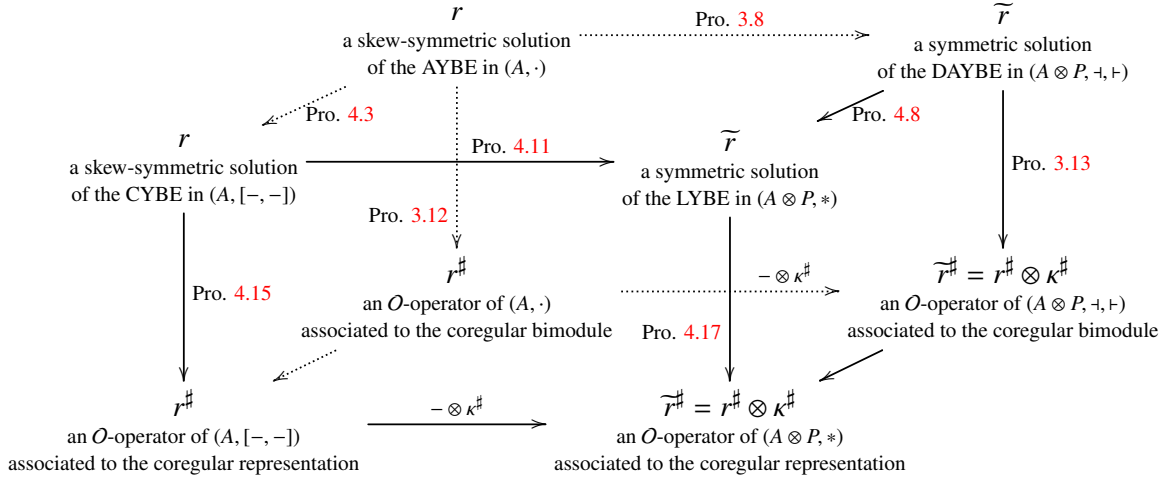
\begin{displaymath}
\xymatrix@R=0.4cm@C=-1cm{
&\txt{{\small $r$}\\ {\tiny a skew-symmetric solution} \\
{\tiny of the $\AYBE$ in $(A, \cdot)$}}
\ar@{.>}[dd]_(0.7){{\rm Pro.}~\ref{pro:o-ass}}
\ar@{.>}[rr]^{{\rm Pro.}~\ref{pro:AYBE-DYBE}}&
&\txt{{\small $\widetilde{r}$} \\  {\tiny a symmetric solution} \\
{\tiny of the $\DAYBE$ in $(A\otimes P, \dashv, \vdash)$}}
\ar@{->}[dd]^{{\rm Pro.}~\ref{pro:o-dia}}\\
\txt{{\small $r$}\\ {\tiny a skew-symmetric solution} \\
{\tiny of the $\CYBE$ in $(A, [-,-])$}}
\ar@{->}[dd]^{{\rm Pro.}~\ref{pro:o-lie}}
\ar@{<.}[ru]_{{\rm Pro.}~\ref{pro:ass-Lie-YBE}}
\ar@{->}[rr]^(0.6){{\rm Pro.}~\ref{pro:CYBE-LYBE}}&
&\txt{{\small $\widetilde{r}$}\\ {\tiny a symmetric solution} \\
{\tiny of the $\LYBE$ in $(A\otimes P, \ast)$}}
\ar@{<-}[ru]_{{\rm Pro.}~\ref{pro:diass-Leib-YBE}}
\ar@{->}[dd]_(0.65){{\rm Pro.}~\ref{pro:o-leib}}\\
&\txt{{\small $r^{\sharp}$}\\{\tiny an $\mathcal{O}$-operator of $(A, \cdot)$} \\
{\tiny associated to the coregular bimodule}} \ar@{.>}[ld] &
&\txt{{\small $\widetilde{r}^{\sharp}=r^{\sharp}\otimes\kappa^{\sharp}$}\\
{\tiny an $\mathcal{O}$-operator of $(A\otimes P, \dashv, \vdash)$} \\
{\tiny associated to the coregular bimodule}}
\ar@{<.}[ll]_(0.4){\tiny\mbox{$-\otimes\kappa^{\sharp}$}}\\
\txt{{\small $r^{\sharp}$}\\{\tiny an $\mathcal{O}$-operator of $(A, [-,-])$} \\
{\tiny associated to the coregular representation}}&
&\txt{{\small $\widetilde{r}^{\sharp}=r^{\sharp}\otimes\kappa^{\sharp}$}\\
{\tiny an $\mathcal{O}$-operator of $(A\otimes P, \ast)$} \\
{\tiny associated to the coregular representation}}
\ar@{<-}[ru]\ar@{<-}[ll]_{\tiny\mbox{$-\otimes\kappa^{\sharp}$}}&}
\end{displaymath}
}\\[-2mm]
where the commutativity of the front and back sides of the diagram following from the
proof of Theorem 3.10 in \cite{HL} and Theorem \ref{thm:indu-sdibia} respectively.
Therefore, we obtain the overall commutative diagram in Section \ref{sec:intr}.

\subsection{Nondegenerate solution of Yang-Baxter equation and symplectic structure}
\label{subsec:sym-symp}
Here we consider the relationship between symplectic structures on Lie
algebras, associative algebras, diassociative algebras and Liebniz algebras.
Some special solutions of the Yang-Baxter equation are closely related to
$\mathcal{O}$-operators, while some special nondegenerate solutions of the Yang-Baxter
equation are closely related to symplectic structure. Recall that a {\bf Connes cocycle}
on an associative algebra $(A, \cdot)$ is a skew-symmetric bilinear form $\omega(-,-)$
satisfying
$$
\omega(a_{1}\cdot a_{2},\; a_{3})+\omega(a_{2}\cdot a_{3},\; a_{1})
+\omega(a_{3}\cdot a_{1},\; a_{2})=0,
$$
for any $a_{1}, a_{2}, a_{3}\in A$. An associative algebra $(A, \cdot)$ with a
nondegenerate Connes cocycle $\omega(-,-)$ is called a {\bf symplectic associative algebra}.
We denote this symplectic associative algebra by $(A, \cdot, \omega)$.

Let $V$ be a vector space and $r\in V\otimes V$ be nondegenerate. Then we can define a
nondegenerate bilinear form $\omega_{r}(-,-)$ on $V$  by
\begin{align}
\omega_{r}(v_{1}, v_{2}):=\langle(r^{\sharp})^{-1}(v_{1}),\; v_{2}\rangle, \label{symp}
\end{align}
for any $v_{1}, v_{2}\in V$. The bilinear form $\omega_{r}(-,-)$ is called
{\bf the induced bilinear form} by $r$. It is easy to see that $\omega_{r}(-,-)$ is symmetric
(reps. skew-symmetric) if $r$ is symmetric (reps. skew-symmetric).

\begin{pro}[\cite{Agu,Bai}]\label{pro:sym-ass}
Let $(A, \cdot)$ be an associative algebra and $r\in A\otimes A$ be skew-symmetric and
nondegenerate. Then $r$ is a solution of the $\AYBE$ in $(A, \cdot)$ if and only if
$(A, \cdot, \omega_{r})$ is a symplectic associative algebra, where
$\omega_{r}$ is given by Eq. \eqref{symp}.
\end{pro}

Recall that a {\bf symplectic Lie algebra} (also called a quasi-Frobenius Lie algebra) $(\g,
[-,-], \omega)$ is a Lie algebra $(\g, [-,-])$ with a skew-symmetric nondegenerate
bilinear form $\omega(-,-): \g\times\g\rightarrow\Bbbk$ satisfying
$$
\omega([g_{1}, g_{2}], g_{3})+\omega([g_{3}, g_{1}], g_{2})+\omega([g_{2}, g_{3}], g_{1})=0,
$$
for any $g_{1}, g_{2}, g_{3}\in\g$. Similarly, we have

\begin{pro}[\cite{CP,Kup}]\label{pro:sym-lie}
Let $(\g, [-,-])$ be a Lie algebra and $r\in \g\otimes\g$ be skew-symmetric and
nondegenerate. Then $r$ is a solution of the $\CYBE$ in $(\g, [-,-])$ if and only if
$(\g, [-,-], \omega_{r})$ is a symplectic Lie algebra, where
$\omega_{r}$ is given by Eq. \eqref{symp}.
\end{pro}

By the definition, direct verification yields:

\begin{pro}\label{pro:sym-ass-lie}
Let $(A, \cdot, \omega)$ be a symplectic associative algebra and $(A, [-,-])$ be the
induced Lie algebra by the commutator. Then $(A, [-,-], \omega)$ is a symplectic Lie algebra,
which is called {\bf the symplectic Lie algebra induced by $(A, \cdot, \omega)$}.
\end{pro}

A {\bf symplectic diassociative algebra} $(D, \dashv, \vdash, \omega)$ is a diassociative
algebra $(D, \dashv, \vdash)$ with a symmetric nondegenerate
bilinear form $\omega(-,-): D\times D\rightarrow\Bbbk$ satisfying
\begin{align*}
&\omega(d_{1},\; d_{2}\vdash d_{3}-d_{2}\dashv d_{3})
+\omega(d_{2},\; d_{3}\dashv d_{1})-\omega(d_{3},\; d_{1}\vdash d_{2})=0,\\
&\omega(d_{1},\; d_{2}\dashv d_{3})+\omega(d_{2},\; d_{3}\dashv d_{1}
-d_{3}\vdash d_{1})-\omega(d_{3},\; d_{1}\dashv d_{2})=0,
\end{align*}
for any $d_{1}, d_{2}, d_{3}\in D$. Now we show that each symmetric
nondegenerate solution of the $\DAYBE$ in $(D, \dashv, \vdash)$ induced a symplectic
structure on $(D, \dashv, \vdash)$.

\begin{pro}\label{pro:sym-dia}
Let $(D, \dashv, \vdash)$ be a diassociative algebra and $r\in D\otimes D$ be symmetric
and nondegenerate. Then $r$ is a solution of the $\DAYBE$ in $(D, \dashv, \vdash)$ if
and only if $(D, \dashv, \vdash, \omega_{r})$ is a symplectic diassociative algebra, where
$\omega_{r}$ is given by Eq. \eqref{symp}.
\end{pro}

\begin{proof}
Since $r\in D\otimes D$ is symmetric, by Proposition \ref{pro:o-dia}, we get that
$r$ is a solution of the $\DAYBE$ in $(D, \dashv, \vdash)$ if and only if
$r^{\sharp}$ is an $\mathcal{O}$-operator of $(D, \dashv, \vdash)$ associated to the
coregular bimodule $(D^{\ast}, \fr_{\vdash}^{\ast}-\fr_{\dashv}^{\ast},
-\fl_{\vdash}^{\ast}, -\fr_{\dashv}^{\ast}, \fl_{\dashv}^{\ast}-\fl_{\vdash}^{\ast})$.
Moreover, since $r^{\sharp}: D^{\ast}\rightarrow D$ is an isomorphism, for any
$d_{1}, d_{2}, d_{3}\in D$, there are $\xi_{1}, \xi_{2}, \xi_{3}\in D^{\ast}$ such that
$r^{\sharp}(\xi_{i})=d_{i}$, $i=1,2,3$. Then we have
\begin{align*}
&\omega_{r}(r^{\sharp}(\xi_{1})\vdash r^{\sharp}(\xi_{2}),\; d_{3})
=\omega_{r}(d_{1}\vdash d_{2},\; d_{3}),\\
&\omega_{r}(r^{\sharp}(\fr_{\dashv}^{\ast}(r^{\sharp}(\xi_{1}))(\xi_{2})),\; d_{3})
=\langle\fr_{\dashv}^{\ast}(r^{\sharp}(\xi_{1}))(\xi_{2}),\; d_{3}\rangle
=-\langle\xi_{2},\; d_{3}\dashv r^{\sharp}(\xi_{1})\rangle
=-\omega_{r}(d_{2},\; d_{3}\dashv d_{1}),\\
&\omega_{r}(r^{\sharp}(\fl_{\vdash}^{\ast}(r^{\sharp}(\xi_{2}))(\xi_{1})),\; d_{3})
=\langle\fl_{\vdash}^{\ast}(r^{\sharp}(\xi_{2}))(\xi_{1}),\; d_{3}\rangle
=-\langle\xi_{1},\; r^{\sharp}(\xi_{2})\vdash d_{3}\rangle
=-\omega_{r}(d_{1},\; d_{2}\vdash d_{3}),\\
&\omega_{r}(r^{\sharp}(\fl_{\dashv}^{\ast}(r^{\sharp}(\xi_{2}))(\xi_{1})),\; d_{3})
=\langle\fl_{\dashv}^{\ast}(r^{\sharp}(\xi_{2}))(\xi_{1}),\; d_{3}\rangle
=-\langle\xi_{1},\; r^{\sharp}(\xi_{2})\dashv d_{3}\rangle
=-\omega_{r}(d_{1},\; d_{2}\dashv d_{3}).
\end{align*}
This means that
$$
r^{\sharp}(\xi_{1})\vdash r^{\sharp}(\xi_{2})
=r^{\sharp}\big(-\fr_{\dashv}^{\ast}(r^{\sharp}(\xi_{1}))(\xi_{2})
-\fl_{\vdash}^{\ast}(r^{\sharp}(\xi_{2}))(\xi_{1})
+\fl_{\dashv}^{\ast}(r^{\sharp}(\xi_{2}))(\xi_{1})\big)
$$
if and only if $\omega_{r}(d_{1}\vdash d_{2},\; d_{3})=\omega_{r}(d_{2},\; d_{3}\dashv d_{1})
+\omega_{r}(d_{1},\; d_{2}\vdash d_{3}-d_{2}\dashv d_{3})$. Similarly, we also have
$$
r^{\sharp}(\xi_{1})\dashv r^{\sharp}(\xi_{2})
=r^{\sharp}\big(\fr_{\vdash}^{\ast}(r^{\sharp}(\xi_{1}))(\xi_{2})
-\fr_{\dashv}^{\ast}(r^{\sharp}(\xi_{1}))(\xi_{2})
-\fl_{\vdash}^{\ast}(r^{\sharp}(\xi_{2}))(\xi_{1})\big)
$$
if and only if $\omega_{r}(d_{1}\dashv d_{2},\; d_{3})=\omega_{r}(d_{1},\; d_{2}\dashv d_{3})
+\omega_{r}(d_{2},\; d_{3}\dashv d_{1}-d_{3}\vdash d_{1})$.
Thus, we obtain that $r$ is a solution of the $\DAYBE$ in $(D, \dashv, \vdash)$ if
and only if $(D, \dashv, \vdash, \omega_{r})$ is a symplectic diassociative algebra.
\end{proof}

A {\bf symplectic Leibniz algebra} $(L, \ast, \omega)$ is a Leibniz algebra
$(L, \ast)$ with a symmetric nondegenerate bilinear form $\omega(-,-): L\times
L\rightarrow\Bbbk$ satisfying
$$
\omega(x_{3},\; x_{1}\ast x_{2})+\omega(x_{2},\; x_{1}\ast x_{3})
-\omega(x_{1},\; x_{2}\ast x_{3})-\omega(x_{1},\; x_{3}\ast x_{2})=0.
$$
for any $x_{1}, x_{2}, x_{3}\in L$. The symplectic structure on a Leibniz algebra
has been studied in \cite{TXS}. Similar to the proof of Proposition \ref{pro:sym-dia},
we can obtain:

\begin{pro}\label{pro:sym-leib}
Let $(L, \ast)$ be a Leibniz algebra and $r\in L\otimes L$ be symmetric
and nondegenerate. Then $r$ is a solution of the $\LYBE$ in $(L, \ast)$ if
and only if $(L, \ast, \omega_{r})$ is a symplectic Leibniz algebra, where
$\omega_{r}$ is given by Eq. \eqref{symp}.
\end{pro}

Moreover, by direct verification, we have

\begin{pro}\label{pro:sym-diass-leib}
Let $(D, \dashv, \vdash, \omega)$ be a symplectic diassociative algebra and
$(D, \ast)$ be the induced Leibniz algebra. Then $(D, \ast, \omega)$ is
a symplectic Leibniz algebra, which is called {\bf the symplectic Leibniz algebra
induced by $(D, \dashv, \vdash, \omega)$}.
\end{pro}

Let $(A, \cdot)$ be an associative algebra and $(P, \diamond, \varpi)$ be a quadratic
perm algebra. Consider the symplectic structure on the induced diassociative algebra
$(A\otimes P, \dashv, \vdash)$, we have

\begin{pro}\label{pro:sym-ass-diass}
Let $(A, \cdot, \omega)$ be a symplectic associative algebra and $(P, \diamond, \varpi)$
be a quadratic perm algebra. If we define a bilinear form $\widetilde{\omega}(-,-)$
on the induced diassociative algebra $(A\otimes P, \dashv, \vdash)$ by
$$
\widetilde{\omega}(a_{1}\otimes p_{1},\; a_{2}\otimes p_{2})
=\omega(a_{1}, a_{2})\varpi(p_{1}, p_{2}),
$$
for any $a_{1}, a_{2}\in A$ and $p_{1}, p_{2}\in P$. Then $(A\otimes P, \dashv, \vdash,
\widetilde{\omega})$ is a symplectic diassociative algebra.
\end{pro}

\begin{proof}
First, it is easy to see that $\widetilde{\omega}(-,-)$ is symmetric and nondegenerate
since $\omega(-,-)$ and $\varpi(-,-)$ are skew-symmetric and nondegenerate.
Since $(A, \cdot, \omega)$ be a symplectic associative algebra and $\varpi(-,-)$ is
invariant, for any $a_{1}, a_{2}, a_{3}\in A$, we have $\omega(a_{1},\; a_{2}\cdot a_{3})
+\omega(a_{2},\; a_{3}\cdot a_{1})+\omega(a_{3},\; a_{1}\cdot a_{2})=0$, and for any
$p_{1}, p_{2}, p_{3}\in P$, $\varpi(p_{2},\; p_{1}\diamond p_{3})=-\varpi(p_{3},\; p_{1}
\diamond p_{2})=\varpi(p_{1},\;p_{2}\diamond p_{3}-p_{3}\diamond p_{2})$. Thus,
\begin{align*}
&\; \widetilde{\omega}(a_{1}\otimes p_{1},\; (a_{2}\otimes p_{2})\vdash(a_{3}\otimes p_{3}))
-\widetilde{\omega}(a_{1}\otimes p_{1},\; (a_{2}\otimes p_{2})
\dashv(a_{3}\otimes p_{3}))\\[-1mm]
&\qquad +\widetilde{\omega}(a_{2}\otimes p_{2},\; (a_{3}\otimes p_{3})\dashv
(a_{1}\otimes p_{1}))
-\widetilde{\omega}(a_{3}\otimes p_{3},\; (a_{1}\otimes p_{1})\vdash
(a_{2}\otimes p_{2}))\\
=&\; \omega(a_{1},\; a_{2}\cdot a_{3})\varpi(p_{1},\; p_{2}\diamond p_{3})
-\omega(a_{1},\; a_{2}\cdot a_{3})\varpi(p_{1},\; p_{3}\diamond p_{2})\\[-1mm]
&\qquad +\omega(a_{2},\; a_{3}\cdot a_{1})\varpi(p_{2},\; p_{1}\diamond p_{3})
-\omega(a_{3},\; a_{1}\cdot a_{2})\varpi(p_{3},\; p_{1}\diamond p_{2})\\
=&\;\Big(\omega(a_{1},\; a_{2}\cdot a_{3})+\omega(a_{2},\; a_{3}\cdot a_{1})
+\omega(a_{3},\; a_{1}\cdot a_{2})\Big)\varpi(p_{1},\; p_{2}\diamond p_{3})\\[-1mm]
&\qquad -\Big(\omega(a_{1},\; a_{2}\cdot a_{3})+\omega(a_{2},\; a_{3}\cdot a_{1})
+\omega(a_{3},\; a_{1}\cdot a_{2})\Big)\varpi(p_{1},\; p_{3}\diamond p_{2})\\
=&\; 0.
\end{align*}
Similarly, we also have
\begin{align*}
&\; \widetilde{\omega}(a_{1}\otimes p_{1},\; (a_{2}\otimes p_{2})\dashv(a_{3}\otimes p_{3}))
+\widetilde{\omega}(a_{2}\otimes p_{2},\; (a_{3}\otimes p_{3})\dashv
(a_{1}\otimes p_{1}))\\[-1mm]
&\qquad -\widetilde{\omega}(a_{2}\otimes p_{2},\; (a_{3}\otimes p_{3})\vdash
(a_{1}\otimes p_{1}))
-\widetilde{\omega}(a_{3}\otimes p_{3},\; (a_{1}\otimes p_{1})\dashv
(a_{2}\otimes p_{2}))=0.
\end{align*}
Thus, $(A\otimes P, \dashv, \vdash, \widetilde{\omega})$ is a symplectic diassociative algebra.
\end{proof}

Let $(A, \cdot)$ be an associative algebra, $(P, \diamond, \varpi)$ be a quadratic
perm algebra and $r=\sum_{i}x_{i}\otimes y_{i}\in A\otimes A$ be a skew-symmetric and
nondegenerate solution of the $\AYBE$ in $(A, \cdot)$. In Proposition \ref{pro:AYBE-DYBE},
we have construct a symmetric nondegenerate solution
$$
\widetilde{r}=\sum_{i, j}(x_{i}\otimes e_{j})\otimes(y_{i}\otimes f_{j})
\in(A\otimes P)\otimes(A\otimes P)
$$
of the $\DAYBE$ in the induced diassociative algebra $(A\otimes P, \dashv, \vdash)$,
where $\{e_{1}, e_{2},\cdots, e_{n}\}$ is a basis of $P$ and $\{f_{1},
f_{2},\cdots, f_{n}\}$ is the dual basis of $\{e_{1}, e_{2},\cdots, e_{n}\}$ with
respect to $\varpi(-,-)$. For the skew-symmetric element $\kappa=\sum_{j}e_{j}\otimes
f_{j}\in P\otimes P$, we have a linear isomorphism $\kappa^{\sharp}: P^{\ast}\rightarrow P$
by $\langle\kappa^{\sharp}(\xi_{1}),\; \xi_{2}\rangle=\langle\kappa,\;
\xi_{1}\otimes\xi_{2}\rangle$, for any $\xi_{1}, \xi_{2}\in P^{\ast}$.
It is easy to see that $\kappa^{\sharp}(e_{i}^{\ast})=f_{i}$ if we denote by
$\{e_{1}^{\ast}, e_{2}^{\ast},\cdots, e_{n}^{\ast}\}$ the dual basis of
$\{e_{1}, e_{2},\cdots, e_{n}\}$ with respect to $\langle-,-\rangle$.
Moreover, using $\kappa$, we can obtain a skew-symmetric bilinear form
$\omega_{\kappa}: P\otimes P\rightarrow\Bbbk$ by $\omega_{\kappa}(p_{1}, p_{2})=
\langle(\kappa^{\sharp})^{-1}(p_{1}),\; p_{2}\rangle$. Note that $\omega_{\kappa}(f_{i},
e_{j})=\delta_{ij}=\varpi(f_{i}, e_{j})$ for any $1\leq i, j\leq n$. We get that
the bilinear form $\omega_{\kappa}(-,-)$ is exactly the original skew-symmetric
nondegenerate bilinear form $\varpi(-,-)$.

\begin{pro}\label{pro:ass-diass-comm}
Let $(A, \cdot)$ be an associative algebra, $(P, \diamond, \varpi)$ be a quadratic perm
algebra and $r=\sum_{i}x_{i}\otimes y_{i}\in A\otimes A$ be a skew-symmetric nondegenerate
solution of the $\AYBE$ in $(A, \cdot)$. Suppose $\widetilde{r}=\sum_{i, j}(x_{i}\otimes
e_{j})\otimes(y_{i}\otimes f_{j})\in(A\otimes P)\otimes(A\otimes P)$ is the symmetric
nondegenerate solution of the $\DAYBE$ in the induced diassociative algebra
$(A\otimes P, \dashv, \vdash)$ and $\omega_{\widetilde{r}}(-,-)$ is the bilinear form
on $(A\otimes P, \dashv, \vdash)$ by Eq. \eqref{symp}. Then we have
$$
\omega_{\widetilde{r}}(a_{1}\otimes p_{1},\; a_{2}\otimes p_{2})
=\omega_{r}(a_{1}, a_{2})\varpi(p_{1}, p_{2}),
$$
for any $a_{1}, a_{2}\in A$ and $p_{1}, p_{2}\in P$. Therefore, we obtain the following
commutative diagram:
$$
\xymatrix@C=3cm@R=0.5cm{
\txt{{\small $r$}\\ {\tiny a nondegenerate skew-symmetric} \\
{\tiny solution of the $\AYBE$ in $(A, \cdot)$}}
\ar[d]_(0.6){{\rm Pro.}~\ref{pro:sym-ass}}\ar[r]^{{\rm Pro.}~\ref{pro:AYBE-DYBE}} &
\txt{{\small $\widetilde{r}$} \\  {\tiny a nondegenerate symmetric solution} \\
{\tiny of the $\DAYBE$ in $(A\otimes P, \dashv, \vdash)$}}
\ar[d]^(0.6){{\rm Pro.}~\ref{pro:sym-dia}} \\
\txt{{\small $(A, \cdot, \omega_{r})$}\\{\tiny a symplectic associative algebra}}
\ar[r]^{{\rm Pro.}~\ref{pro:sym-ass-diass}}
& \txt{{\small $(A\otimes P, \dashv, \vdash, \omega_{\widetilde{r}})$}\\
{\tiny a symplectic diassociative algebra}}}
$$
\end{pro}

\begin{proof}
Since $\widetilde{r}$ is a symmetric nondegenerate solution of the $\DAYBE$ in
$(A\otimes P, \dashv, \vdash)$, by Eq. \eqref{symp}, we get a bilinear form
$\omega_{\widetilde{r}}(-,-)$ on $A\otimes P$. By direct calculation, we have
$$
\omega_{\widetilde{r}}(a_{1}\otimes p_{1},\; a_{2}\otimes p_{2})
=\langle\eta_{1}\otimes\xi_{1},\; a_{2}\otimes p_{2}\rangle
=\langle\eta_{1}, a_{2}\rangle\langle\xi_{1}, p_{2}\rangle
=\omega_{r}(a_{1}, a_{2})\omega_{\kappa}(p_{1}, p_{2}),
$$
where $\eta_{1}\in A^{\ast}$ satisfying $r^{\sharp}(\eta_{1})=a_{1}$
and $\xi_{1}\in P^{\ast}$ satisfying $\kappa^{\sharp}(\xi_{1})=p_{1}$.
Note that $\omega_{\kappa}(p_{1}, p_{2})=\varpi(p_{1}, p_{2})$.
We obtain that $\omega_{\widetilde{r}}(a_{1}\otimes p_{1},\; a_{2}\otimes p_{2})
=\omega_{r}(a_{1}, a_{2})\varpi(p_{1}, p_{2})$.
That is, the bilinear form on $A\otimes P$ constructed in Proposition \ref{pro:sym-ass-diass}
by $\omega_{r}(-,-)$ is exactly the bilinear form on $A\otimes P$ induced by
$\widetilde{r}$. Hence, we get the commutative diagram.
\end{proof}

Similar to the proof of Proposition \ref{pro:sym-ass-diass}, we have

\begin{pro}\label{pro:sym-lie-leib}
Let $(\g, [-,-], \omega)$ be a symplectic Lie algebra and $(P, \diamond, \varpi)$
be a quadratic perm algebra. If we define a bilinear form $\widetilde{\omega}(-,-)$
on the induced Leibniz algebra $(\g\otimes P, \ast)$ by
$$
\widetilde{\omega}(g_{1}\otimes p_{1},\; g_{2}\otimes p_{2})
=\omega(g_{1}, g_{2})\varpi(p_{1}, p_{2}),
$$
for any $g_{1}, g_{2}\in\g$ and $p_{1}, p_{2}\in P$. Then $(\g\otimes P, \ast,
\widetilde{\omega})$ is a symplectic Leibniz algebra.
\end{pro}


Thus, starting from a symplectic associative algebra, we have obtained two approaches to
constructing symplectic Leibniz algebras, and the symplectic Leibniz algebras obtained
from these two approaches are consistent.

\begin{thm}\label{thm:sym-ass-leib}
Let $(A, \cdot, \omega)$ be a symplectic associative algebra and $(P, \diamond, \varpi)$
be a quadratic perm algebra. Suppose $(A\otimes P, \dashv, \vdash)$ is the diassociative
algebra induced from $(A, \cdot)$ and $(P, \diamond)$, $(A, [-,-])$ is the Lie algebra
induced from $(A, \cdot)$ and $(A\otimes P, \ast)$ is the Leibniz algebra induced form
$(A, [-,-])$ and $(P, \diamond)$. Then we have the following commutative diagram about
symplectic associative algebra, symplectic diassociative algebra, symplectic Lie
algebra and symplectic Leibniz algebra:
$$
\xymatrix@C=2cm@R=0.7cm{
\txt{$(A, \cdot, \omega)$ \\
{\tiny a symplectic associative algebra}}
\ar[d]_{{\rm Pro.}~\ref{pro:sym-ass-lie}} \ar[r]^{{\rm Pro.}~\ref{pro:sym-ass-diass}}
&\txt{$(A\otimes P, \dashv, \vdash, \widetilde{\omega})$\\
{\tiny a symplectic diassociative algebra}}\ar[d]^{{\rm Pro.}~\ref{pro:sym-diass-leib}} \\
\txt{$(A, [-,-], \omega)$ \\ {\tiny a symplectic Lie algebra}}
\ar[r]^{{\rm Pro.}~\ref{pro:sym-lie-leib}}
& \txt{$(A\otimes P, \ast, \widetilde{\omega})$ \\ {\tiny a symplectic Leibniz algebra}}}
$$
where $\widetilde{\omega}(a_{1}\otimes p_{1},\; a_{2}\otimes p_{2})
=\omega(a_{1}, a_{2})\varpi(p_{1}, p_{2})$ for any $a_{1}, a_{2}\in A$
and $p_{1}, p_{2}\in P$.
\end{thm}

Moreover, similar to the proof of Proposition \ref{pro:ass-diass-comm}, we also have

\begin{pro}\label{pro:lie-leib-comm}
Let $(\g, [-,-])$ be a Lie algebra, $(P, \diamond, \varpi)$ be a quadratic perm algebra
and $r=\sum_{i}x_{i}\otimes y_{i}\in\g\otimes\g$ be a skew-symmetric nondegenerate solution
of the $\CYBE$ in $(\g, [-,-])$. Suppose $\widetilde{r}=\sum_{i, j}(x_{i}\otimes e_{j})
\otimes(y_{i}\otimes f_{j})\in(\g\otimes P)\otimes(\g\otimes P)$ is the symmetric
nondegenerate solution of the $\LYBE$ in the induced Leibniz algebra $(\g\otimes P, \ast)$
and $\omega_{\widetilde{r}}(-,-)$ is the bilinear form on $(\g\otimes P, \ast)$ by
Eq. \eqref{symp}. Then we have
$$
\omega_{\widetilde{r}}(g_{1}\otimes p_{1},\; g_{2}\otimes p_{2})
=\omega_{r}(g_{1}, g_{2})\varpi(p_{1}, p_{2}),
$$
for any $g_{1}, g_{2}\in\g$ and $p_{1}, p_{2}\in P$.
\end{pro}

Therefore, by Propositions \ref{pro:sym-ass}-\ref{pro:lie-leib-comm}, we obtain the
following commutative diagram of the nondegenerate solutions of the Yang-Baxter equation
and the symplectic structure on the corresponding algebras:
{\small
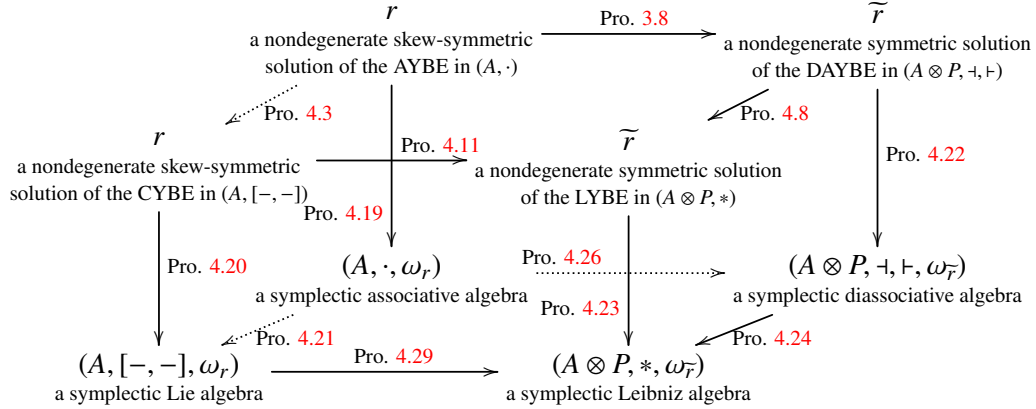
\begin{displaymath}
\xymatrix@R=0.4cm@C=-1cm{
&\txt{{\small $r$}\\ {\tiny a nondegenerate skew-symmetric} \\
{\tiny solution of the $\AYBE$ in $(A, \cdot)$}}
\ar@{->}[dd]_(0.75){{\rm Pro.}~\ref{pro:sym-ass}}
\ar@{->}[rr]^{{\rm Pro.}~\ref{pro:AYBE-DYBE}}&
&\txt{{\small $\widetilde{r}$} \\  {\tiny a nondegenerate symmetric solution} \\
{\tiny of the $\DAYBE$ in $(A\otimes P, \dashv, \vdash)$}}
\ar@{->}[dd]^{{\rm Pro.}~\ref{pro:sym-dia}}\\
\txt{{\small $r$}\\ {\tiny a nondegenerate skew-symmetric} \\
{\tiny solution of the $\CYBE$ in $(A, [-,-])$}}
\ar@{->}[dd]^{{\rm Pro.}~\ref{pro:sym-lie}}
\ar@{<.}[ru]_{{\rm Pro.}~\ref{pro:ass-Lie-YBE}}
\ar@{->}[rr]^(0.6){{\rm Pro.}~\ref{pro:CYBE-LYBE}}&
&\txt{{\small $\widetilde{r}$}\\ {\tiny a nondegenerate symmetric solution} \\
{\tiny of the $\LYBE$ in $(A\otimes P, \ast)$}}
\ar@{<-}[ru]_{{\rm Pro.}~\ref{pro:diass-Leib-YBE}}
\ar@{->}[dd]_(0.7){{\rm Pro.}~\ref{pro:sym-leib}}\\
&\txt{{\small $(A, \cdot, \omega_{r})$}\\{\tiny a symplectic associative algebra}}
\ar@{.>}[ld]^{{\rm Pro.}~\ref{pro:sym-ass-lie}} &
&\txt{{\small $(A\otimes P, \dashv, \vdash, \omega_{\widetilde{r}})$}\\
{\tiny a symplectic diassociative algebra}}
\ar@{<.}[ll]_(0.65){{\rm Pro.}~\ref{pro:ass-diass-comm}}\\
\txt{{\small $(A, [-,-], \omega_{r})$}\\ {\tiny a symplectic Lie algebra}}&
&\txt{{\small $(A\otimes P, \ast, \omega_{\widetilde{r}})$}\\
{\tiny a symplectic Leibniz algebra}}
\ar@{<-}[ru]_{{\rm Pro.}~\ref{pro:sym-diass-leib}}
\ar@{<-}[ll]_{{\rm Pro.}~\ref{pro:lie-leib-comm}}&}
\end{displaymath}
}

\begin{ex}\label{ex:ind-symdiam}
Let $(A, \cdot)$ be the $2$-dimensional associative algebra given in Example
\ref{ex:ind-tribi}, $(P, \diamond, \varpi)$ be the $2$-dimensional quadratic perm
algebra given in Example \ref{ex:dual-perm}, $(A\otimes P, \dashv, \vdash)$ be the
$4$-dimensional diassociative algebra given in Example \ref{ex:ind-tribi},
$(A, [-,-])$ and $(A\otimes P, \ast)$ be the Lie algebra and Leibniz algebra given in
Example \ref{ex:ind-tridiam} respectively. Denote by $r=e_{2}\otimes e_{1}-e_{1}\otimes
e_{2}\in A\otimes A$. Then $r$ is a nondegenerate skew-symmetric solution of the
$\AYBE$ in $(A, \cdot)$ and the element $\widetilde{r}$ given in Example \ref{ex:ind-tribi}
is a nondegenerate symmetric solution of the $\DAYBE$ in $(A\otimes P, \dashv, \vdash)$.
Thus, $r$ induces a skew-symmetric bilinear form $\omega_{r}: A\times A\rightarrow\Bbbk$:
$$
\omega_{r}(e_{1}, e_{2})=1;
$$
and $\widetilde{r}$ induces a symmetric bilinear form $\omega_{\widetilde{r}}:
(A\otimes P)\times(A\otimes P)\rightarrow\Bbbk$:
$$
\omega_{\widetilde{r}}(e_{1}\otimes x_{1},\; e_{2}\otimes x_{2})=1,\qquad\qquad
\omega_{\widetilde{r}}(e_{1}\otimes x_{2},\; e_{2}\otimes x_{1})=-1.
$$
Then it is easy to see that $\omega_{\widetilde{r}}(e_{i}\otimes x_{l}, e_{j}\otimes x_{k})
=\omega_{r}(e_{i}, e_{j})\varpi(x_{l}, x_{k})$ for any $1\leq i, j, l, k\leq2$.
One can check that $(A, \cdot, \omega_{r})$, $(A, [-,-], \omega_{r})$, $(A\otimes P,
\dashv, \vdash)$ and $(A\otimes P, \dashv, \vdash)$ are respectively symplectic
associative algebra, symplectic Lie algebra, symplectic diassociative algebra and
symplectic Leibniz algebra, and the commutative diagram above holds.
\end{ex}

\bigskip
\noindent
{\bf Acknowledgements. } This work was financially supported by National
Natural Science Foundation of China (No.11771122).

\smallskip
\noindent
{\bf Declaration of interests.} The authors have no conflicts of interest to disclose.

\smallskip
\noindent
{\bf Data availability.} Data sharing is not applicable to this article as no new data were
created or analyzed in this study.

 \end{document}